%% file: similarity.tex
\documentclass[10pt,a4paper]{amsart}
\pdfoutput=1
\usepackage[english]{babel}
\usepackage[english]{translator}
\usepackage[T1]{fontenc}
\usepackage[utf8]{inputenc}

\usepackage{amsmath,amsthm,amssymb,amsfonts}
\usepackage{bbm}\newcommand{\1}{{\mathbbm{1}}}
\usepackage[foot]{amsaddr}

\usepackage{graphicx}
\usepackage{subfigure}
\usepackage{psfrag}

\usepackage{enumitem}
\usepackage{mathtools}

\usepackage[colorlinks=true]{hyperref}
\hypersetup{
  pdffitwindow=false,
  pdfhighlight=/O,
  pdfnewwindow,
  colorlinks=true,
  citecolor=red,            
  linkcolor=blue,             
  menucolor=blue,            
  urlcolor=blue,             
  pdfpagemode=UseOutlines,
  bookmarksnumbered=true,
  linktocpage,
  pdftitle={Freezing similarity solutions in multi-dimensional Burgers' Equation},
  pdfsubject={Freezing similarity solutions in multi-dimensional Burgers' Equation},
  pdfauthor={Jens Rottmann-Matthes},
  pdfkeywords={similarity solutions, freezing method, relative equilibria, Burgers' equation},
  pdfcreator={pdflatex,bibtex},
  pdfproducer={LaTeX mit hyperref}
}

\usepackage[
  hmarginratio={1:1},     
  vmarginratio={1:1},     
  textwidth=450pt,        
  heightrounded,          
]{geometry}

\newcommand{\R}{\ensuremath{\mathbb{R}}}

\newcommand{\N}{\ensuremath{\mathbb{N}}}
\newcommand{\Q}{\ensuremath{\mathbb{Q}}}
\newcommand{\T}{\ensuremath{\mathrm{T}}}

\newcommand{\vect}[1]{\begin{pmatrix} #1%
\end{pmatrix}}

\newcommand{\wt}[1]{\widetilde{#1}}
\newcommand{\wh}[1]{\widehat{#1}}

\DeclareMathOperator{\linearhull}{span}

\DeclareMathOperator{\supp}{supp}

\DeclareMathOperator{\argmin}{argmin}
\DeclareMathOperator{\divergence}{div}
\DeclareMathOperator{\gradient}{\nabla}
\DeclareMathOperator{\laplace}{\Delta}

\newcommand{\dd}[1]{\frac{\partial}{\partial #1}}
\newcommand{\SO}{\ensuremath{\mathrm{SO}}}
\newcommand{\so}{\mathfrak{so}}
\newcommand{\GL}{\mathrm{GL}}
\newcommand{\Hess}{\mathrm{Hess}}
\newcommand{\e}{\mathrm{e}}
\newcommand{\const}{\mathrm{const}}
\newcommand{\scal}{\mathrm{scal}}

\renewcommand{\O}{{\mathrm{ O}}}
\DeclareMathOperator{\tr}{tr}
\DeclareMathOperator{\sgn}{sgn}

\newcommand{\sol}[1]{{\underline{#1}}}
\newcommand{\CC}{\mathcal{C}}
\newcommand{\liealg}[1]{\ensuremath{\mathfrak{#1}}}

\theoremstyle{plain}
\newtheorem*{definition*}{Definition}%
\newtheorem{definition}{Definition}[section]
\newtheorem{theorem}[definition]{Theorem}
\newtheorem{proposition}[definition]{Proposition}
\newtheorem{lemma}[definition]{Lemma}
\newtheorem*{lemma*}{Lemma}
\newtheorem{corollary}[definition]{Corollary}
\newtheorem*{assumption*}{Hypothesis}

\theoremstyle{definition}
\newtheorem*{remark*}{Remark}
\newtheorem{remark}[definition]{Remark}
\newtheorem*{remarks*}{Remarks}
\newtheorem{remarks}[definition]{Remarks}
\newtheorem{example}[definition]{Example}

\allowdisplaybreaks

\begin{document}
\title[Freezing similarity solutions in multi-dimensional Burgers' Equation]{Freezing similarity solutions in multi-dimensional Burgers' Equation}
\setlength{\parindent}{0pt}
\begin{center}
\normalfont\huge\bfseries{\shorttitle}\\
\vspace*{0.25cm}
\end{center}

\vspace*{0.8cm}
\noindent
\hspace*{.33\textwidth}
\begin{minipage}[t]{0.4\textwidth}
\noindent
\textbf{Jens Rottmann-Matthes}\footnotemark[1] \\
Institut für Analysis \\
Karlsruhe Institute of Technology \\
76131 Karlsruhe \\
Germany
\end{minipage}~\\
\footnotetext[1]{e-mail: \textcolor{blue}{jens.rottmann-matthes@kit.edu}, phone: \textcolor{blue}{+49 (0)721 608 41632}, \\
supported by CRC 1173 'Wave Phenomena: Analysis and Numerics', Karlsruhe Institute of Technology}
~\vspace*{0.6cm}

\noindent
\hspace*{.33\textwidth}
Date: October 12, 2016
\normalparindent=12pt

\vspace{0.4cm}
\begin{center}
\begin{minipage}{0.9\textwidth}
  {\small
  \textbf{Abstract.} 
  The topic of this paper are similarity solutions occurring in
  multi-dimensional Burgers' equation.  We present a simple derivation of
  the symmetries appearing in a family of generalizations of Burgers'
  equation in $d$-space dimensions.  These symmetries we use to derive an
  equivalent partial differential algebraic equation (freezing system)
  that allows us to do long time simulations and obtain good approximations
  of similarity solutions by direct forward simulation.  The method also
  allows us without further effort to observe meta-stable behavior near
  N-wave-like patterns.
  }
\end{minipage}
\end{center}

\noindent
\textbf{Key words.} Similarity solutions, relative equilibria, Burgers'
equation, freezing method, scaling symmetry, meta-stable behavior.

\noindent
\textbf{AMS subject classification.} 65P40, 35B40, 35B06, 37C80 (65M99, 35L65)

\input{sec1.tex}

\input{sec2.tex}

\input{sec3.tex}

\input{sec4.tex}

\input{sec5.tex}

\bibliographystyle{abbrv}
\bibliography{LiteraturNeu.bib}

\end{document}

%% file: sec1.tex
\section{Introduction}
\label{sec:intro}
Patterns are abundant in partial differential equations and they often
have an important implication on the interpretation of the system's
behavior.  A simple type of pattern are relative equilibria and the
possibly simplest, non-trivial relative equilibria in partial
differential equations are traveling waves.  Traveling waves for
example appear in the case of the nerve-axon-equations (e.g. in the
Hodgkin-Huxley system) and can be interpreted as the transport of
information.  From this interpretation, it is obvious, that one is not
only interested in the waves shape, but also, how fast it actually
travels and how it evolves in the first place, thus, how fast the
information is actually passed on from its ignition.

One possibility to calculate this by a direct forward simulation, is
the method of freezing, independently introduced in
\cite{BeynThuemmler:2004} and
\cite{RowleyKevrekidisMarsdenLust:2003}, see also
\cite{BeynOttenRottmann:2013}.  The method not only allows to capture
traveling waves by simple long-time simulations, but can also be used for other
relative equilibria such as rotating waves or scroll waves.  
In this article we show how the method can be used to do long-time simulations
of the multi-dimensional Burgers' equation and how similarity solutions of
Burgers' equation can be obtained in this way.
We do not consider the most general and abstract version of the method but only
consider it in such generality as needed for the specific case of Burgers'
equation.  We refer to the dissertation \cite{Rottmann:2010} and also
\cite{BeynOttenRottmann:2013}, which is based on
\cite{BeynThuemmler:2004} and \cite{RowleyKevrekidisMarsdenLust:2003}
for more on the abstract method.  But also note that as in
\cite{RowleyKevrekidisMarsdenLust:2003} and unlike
\cite{BeynThuemmler:2004}, we include a time-scaling.  This will be
crucial for long time simulations and the ability to observe meta-stable
behavior in Burgers' equation.

Burgers' equation,
\begin{equation}
  \label{eq:burgers1d}
  u_t+\left(\tfrac{1}{2}u^2\right)_x=\nu u_{xx},\quad x\in\R,
\end{equation}
was originally introduced by J.M. Burgers
(e.g. \cite{Burgers:1948}) in 1948 as a model of turbulence.  In
\cite{Burgers:1948} he points out that the combination of the dissipative term
$\nu u_{xx}$ and the nonlinear term $uu_x$ characterizes the mechanism of
producing turbulence.  Equation \eqref{eq:burgers1d} is one of the simplest
truly nonlinear partial differential equations and it is well-known that it
produces shock solutions in the inviscid case ($\nu=0$).  Therefore, it is also
frequently used as a test equation for numerical schemes for conservation laws
and for shock-capturing schemes. 
We also consider the following 
multi-dimensional generalizations of Burgers' equation,
\begin{equation}
  \label{eq:burgersdd}
  u_t+\tfrac{1}{p}\divergence(a|u|^p)=\nu\Delta u,\quad x\in\R^d,
\end{equation}
where $d\ge 1$, $a\in\R^d\setminus\{0\}$, $\nu>0$ and $p>1$ are fixed. The
special case of $p=2$ we call the $d$-dimensional Burgers' equation and we call
the case $p=\tfrac{d+1}{d}$ the conservative Burgers' equation.  Note that for
$d=1$ and $a=1$ both special cases reduce to the standard Burgers' equation
\eqref{eq:burgers1d}.  Equation~\eqref{eq:burgersdd} with $p=2$ appears as 
special cases of the multidimensional Burgers' equation
\begin{equation}\label{eq:multiburgers}
  \partial_t \vec{u}+(\vec{u}\cdot\gradient) \vec{u}=\nu\laplace \vec{u},
\end{equation}
which has applications in different areas of physics, see the review
\cite{BecKhanin:2007}.  For example, if
$\vec{u}_1=u$ and $\vec{u}_j=0$ for $j=2,\dots,d$, \eqref{eq:multiburgers} leads
to \eqref{eq:burgersdd} with $a=\e_1$.

We use the symmetries inherent in \eqref{eq:burgersdd} to
split the evolution of the solution into a part that captures the evolution of
the profile and a part that captures the evolution of the solution in the
symmetry group.  This transforms the Cauchy problem for the original partial
differential equation (PDE) \eqref{eq:burgersdd} into an equivalent partial
differential algebraic equation (PDAE), the so-called freezing system.
Stationary solutions to the freezing system PDAE are similarity solutions of
\eqref{eq:burgersdd},  and they not only include the profile of these but also
the full information about the evolution of this profile in the symmetry group.
Moreover, if the similarity solutions are asymptotically stable, they can be
obtained by a direct forward simulation of this PDAE system.  This has been
proved rigorously in several cases, e.g.\ see \cite{Rottmann:2012a}.

From an analytic point of view, the most interesting scale of
viscosities $\nu$ in \eqref{eq:burgers1d} is $\nu\approx 0$.  In
this region the solution to the Cauchy problem for
\eqref{eq:burgers1d} exhibits a metastable behavior in the sense, that
it rapidly approaches a similarity solution of the inviscid
problem, namely an $N$-wave and then has a very long transient until
it finally reaches a true similarity solution of the parabolic
problem, a so called viscosity wave.  This has first been observed
numerically and analyzed in \cite{KimTzavaras:2001} and from a
dynamical systems point of view in \cite{BeckWayne:2009} by means of
the Cole-Hopf-transform.  In both articles the authors have used 
correct asymptotic similarity variables, which can explicitly be calculated for
Burgers' equation.
In contrast to \cite{KimTzavaras:2001} and \cite{BeckWayne:2009}, we take the
point of view that in more complicated equations these similarity variables are
not known a priori and can only be obtained by a numerical calculation.
In fact, our method does not need this information but actually calculates a
suitable choice of similarity variables on the fly. 
The similarity variables obtained by our method may differ from 
the simple scaling variables used in \cite{KimTzavaras:2001} and
\cite{BeckWayne:2009}.  This is because we do not center the solution at $0$
and, moreover, not only allow scalings, but also a
non-zero velocity in the similarity variables.

The plan of the paper is as follows.
In Section~\ref{sec:can_sym} we use a transformation of the coordinates to write the
generalized Burgers' equation \eqref{eq:burgersdd} in a simple canonical form.
For this simple form we then derive a continuous family of symmetries inherent
in the equation which can be used for the numerical calculation.
In Section~\ref{sec:act_gen} we show that the symmetry group obtained in
Section~\ref{sec:can_sym} acts strongly continuous on various function spaces.
This allows us to calculate the generators of the group action on suitable
function spaces.
In Section~\ref{sec:freeze} we then make the ansatz that the evolution of the
solution to \eqref{eq:burgersdd} can be split up into an evolution of the
profile and an evolution along the group orbit to derive the equation in a
co-moving frame.  This ansatz introduces new unknowns into the equation. 
For
example in the case $d=2$, $p=\tfrac{3}{2}$, and $a=\e_1$ in \eqref{eq:burgersdd}
this leads to the following under-determined PDE 
\[v_\tau=\nu \laplace v-v\,v_x+\mu_1 \bigl((xv)_x+(yv)_y)+\mu_2 v_x+\mu_3 v_y\]
for $v$, $\mu_1$, $\mu_2$, and $\mu_3$.
We present two possible choices of phase-conditions to cope with these
artificially
introduced degrees of freedom.  The resulting equation is then the freezing PDAE
and we show that it is equivalent to the original problem.
In the final Section~\ref{sec:exp} we present the numerical results of several
experiments.  These show the ability of the method to calculate similarity
solutions of the multi-dimensional Burgers' equation by direct forward
simulation.  With our method we are also able to observe directly a meta-stable
behavior of the solution not only in the 1d-case as in \cite{KimTzavaras:2001},
but also for the multi-dimensional Burgers' equation with small viscosity
in the vicinity of so called N-wave like patterns.  To our knowledge, this is
the first time, this has been observed in the multi-dimensional case.

Closest to our approach is the article \cite{RowleyKevrekidisMarsdenLust:2003}
where a variant of the method was introduced for different examples, including
the standard 1d Burgers' equation.  But note that in that article neither the
generalization to multi-dimensional problems or more general symmetries as
in \eqref{eq:burgersdd} were considered, nor the long-time or
meta-stable behavior was observed.

Throughout this article we use the following notations and spaces.  We denote
the Euclidean norm in a finite dimensional space by $|\cdot|$ and inner products in $\R^d$ we
write as $\langle\cdot,\cdot\rangle$.  For vectors $x\in\R^d$, $x^\top$ denotes
the transpose, so that $\langle x,y\rangle=x^\top y$ for all $x,y\in\R^d$.  We
interchangeably use the notations $\dd{x_j}u=\partial_{x_j}u=\partial_j u=u_{x_j}$ for the
derivative with respect to $x_j$.  By $D f$ we denote the total derivative (or
Jacobian) of a function $f$ and $D^\alpha f$ denotes the partial derivative
$\partial_{x_1}^{\alpha_1}\dots\partial_{x_d}^{\alpha_d} f$, where $\alpha\in
\N^d$ is a multi-index.  For differentiable maps between manifolds, $f:M\to N$,
we denote by $T_{u_0}f:\T_{u_0}M\to \T_{f(u_0)}N, v\mapsto T_{u_0}f[v]$ the
tangential of the map $f$ at the point $u_0\in M$.

Since we are primarily interested in the approximation of localized (similarity)
solutions, we consider functions, vanishing at infinity in a suitable sense.
Therefore, we consider the space of $k$-times continuously differentiable
functions with compact support,
\[\mathcal{C}_c^k(\R^d)=\left\{u\in\mathcal{C}^k(\R^d):
\supp(u)\subset\R^d \text{ is compact}\right\},\]
and the space of $k$-times continuously differentiable functions vanishing at
infinity, 
\[\mathcal{C}_0^k(\R^d)= \left\{u\in\mathcal{C}^k(\R^d):
  \lim_{|x|\to\infty} |D^\alpha u(x)|=0\;\forall \alpha\in \N^d,
|\alpha|\leq k\right\}.\]
The space $\mathcal{C}_0^k(\R^d)$ with norm
$\|u\|_{\mathcal{C}^k}:=\max_{|\beta|\leq k} \|D^\beta u\|_\infty$ is a
Banach space.  Another suitable choice are the
$L^2$-Sobolev-spaces $H^k(\R^d)$ with the norm
$\|u\|_{H^k}^2=\sum_{|\alpha|\leq k} \|D^\alpha u\|_{L^2}^2$ which is
a Hilbert space for each $k\in\N$.  In $H^k(\R^d)$ we also use the
equivalent norm
\begin{equation}\label{eq:hkdotnorm}
  \|u\|_{\dot{H}^k}:=\|\langle \xi\rangle^k \wh{u}\|_{L^2},
\end{equation}
where $\langle\xi\rangle=\sqrt{1+|\xi|^2}$ and
\[\wh{u}(\xi)=(2\pi)^{-d/2} \int_{\R^d} e^{-ix^\top\xi} u(x)\,dx\]
is the Fourier transform of $u$ (cf.~\cite[Ch.7]{Rudin:1991}). 

\textbf{Acknowledgement}
We gratefully acknowledge financial support by the Deutsche
Forschungsgemeinschaft (DFG) through CRC 1173.


%% file: sec2.tex
\section{Canonical Form and Symmetries of Generalized Burgers' Equation}
\label{sec:can_sym}
In this section we first derive a standard form of the generalized Burgers'
equation~\eqref{eq:burgersdd}.  Then we will calculate certain symmetries for
this standard form.
It is well-known and already appears in \cite{Hopf:1950} that Burgers'
equation~\eqref{eq:burgers1d} exhibits several symmetries.  The system has often
served as an example to illustrate similarity methods, e.g.\ see
\cite[Ex.~6.1]{Olver:1979} where a five-dimensional Lie algebra
corresponding to the symmetries in Burgers' equation is calculated and
also similarity solutions are obtained.  Here we do not follow the
abstract approach from \cite{Olver:1986} but rather directly supply
suitable symmetries, which we will later on use in our numerical method.  For
example, we ignore the shift equivariance with respect to time, which is of no
use to us, since we are interested in the behavior of solutions to the
corresponding Cauchy problems.

\subsection{Canonical Form}
To write equation~(\ref{eq:burgersdd}) in a standard form, we use the
following lemma.
\begin{lemma}
  \label{lem:canform}
  Let $M\in\R^{d,d}$ be invertible and identify $M$ with the linear
  mapping in $\R^d$ it induces.  Furthermore let $1\leq p<\infty$.
  \begin{enumerate}
  \item If $f:\R\to\R$ is locally Lipschitz continuous, it follows
    for $u\in
    W^{1,p}(\R^d,\R)\cap L^\infty(\R^d)$ with $f\circ u\in L^p$:
    $f\circ u\circ M\in W^{1,p}(\R^d)$ and
    \begin{equation}\label{eq:canfrom1}
      D\bigl(f\circ u\circ M\bigr)(x)\xi = Df\bigl(u(Mx)\bigr) Du(Mx)
      M\xi
    \end{equation}
    for almost every $x\in\R^d$ and all $\xi\in\R^d$.
  \item If $u\in W^{2,p}(\R^d)$, then $u\circ M\in W^{2,p}(\R^d)$ and
    \begin{equation}\label{eq:canform2}
      \Delta \bigl(u\circ M\bigr)(x) = \sum_{i,j,k}
      M_{ij}\partial_i\partial_k u(Mx) M_{kj} = \tr\bigl(M^\top
      \Hess(u) M\bigr)(Mx)
    \end{equation}
    for almost every $x\in \R^d$.
  \end{enumerate}
\end{lemma}
The chain-rule immediately implies that \eqref{eq:canfrom1},
\eqref{eq:canform2} hold for $u\in \mathcal{C}^2(\R^d)$.  The
$L^p$-case considered in Lemma~\ref{lem:canform} follows from~\cite[\S
2]{Ziemer:1989}.  The above lemma shows that for $X=\CC_0^0(\R^d)$ and
$Y=\CC^2_0(\R^d)$, respectively $X=L^2(\R^d)$ and $Y=H^2(\R^d)$, a
function
\[u\in\mathcal{C}\bigl([0,T);Y\bigr)
\cap\mathcal{C}^1\bigl((0,T);X\bigr)\]
solves~(\ref{eq:burgersdd})
if and only if
\[v\in\mathcal{C}\bigl([0,|a| T);Y\bigr)
\cap\mathcal{C}^1\bigl((0,|a|T);X\bigr),\]
given by
\[u(x,t)=v\bigl(Q_a x,\tfrac{1}{|a|}t\bigr)=v(y,s),\quad
q=a-|a|\mathrm{e}_1,\quad
Q_a=I-2\frac{q\,q^\top}{q^\top q},\]
satisfies
\begin{equation}\label{eq:burgersdd2}
  \tfrac{\partial}{\partial s}v
  +\tfrac{1}{p}\tfrac{\partial}{\partial
    y_1}\left(|v|^p\right)=\tfrac{\nu}{|a|}\Delta  v. 
\end{equation}
Here we say that a function solves \eqref{eq:burgersdd} or \eqref{eq:burgersdd2} if
the equalities hold for all $t$ as equalities in $X$.
In the sequel we incorporate the
factor $\tfrac{1}{|a|}$ into the viscosity $\nu$, so that we restrict
to
\begin{equation}
  \label{eq:burgersdd'}
  u_t=\nu \Delta u +\frac{1}{p}\frac{\partial}{\partial
    x_1}\bigl(|u|^p\bigr)=: F(u).
\end{equation}
\begin{remark}\label{rem:wellposed}
  A simple well-posedness result can be found e.g. in
  \cite[\S8~Thm.~3.5]{Pazy:1983} in the sense that for every initial data
  $u_0\in H^2$ there exists a unique strong solution.  See also
  \cite[\S~7.2.5]{Pazy:1983}.
\end{remark}
From now on we do not discuss the well-posedness of the problems, but
require that the solution exists and belongs to the function space at hand
without further notice.

\subsection{Symmetries for Burgers' Equation}
To motivate the ansatz, we remark that the nonlinear operator
$F:u\mapsto \nu\Delta u-\tfrac{1}{p}\bigl(|u|^p\bigr)_{x_1}$ is
equivariant with respect to spatial translations in every direction
and with respect to rotations that leave the first coordinate axis
fixed.  Moreover, each of the summands has a scaling property and we
obtain a scaling symmetry by matching these.  More precisely, we make
the ansatz
\begin{equation}
  \label{eq:ansatz}
  u(x)=\frac{1}{\alpha} v\bigl(M^{-1}(x-b)\bigr),
\end{equation}
where $\alpha>0$, $b\in\R^d$, and $M\in\R^{d,d}$ is some invertible
matrix.  Using the chain rule, respectively Lemma~\ref{lem:canform}, easily
follows:
\begin{proposition}
  \label{prop:symmetries}
  Let $p>1$ be fixed. 
  For all $u\in \mathcal{C}^2(\R^d)$ (resp. $u\in H^2(\R^d)\cap
  L^\infty(\R^d)$)
  and $v$ given by~(\ref{eq:ansatz}), holds
  \begin{equation}\label{eq:2.2.1}
    F(u)= \const_{(\alpha,M,b)}\frac{1}{\alpha}
    F(v)\bigl(M^{-1}(\cdot-b)\bigr)
  \end{equation}
  as an equality in
  $\mathcal{C}^0$ (resp. $L^2$) if and only if
  $M=\alpha^{p-1}\begin{pmatrix}1&0\\0&Q
  \end{pmatrix}$, where $Q\in \mathcal{O}(d-1)$, and
  $\const_{(\alpha,M,b)}=\alpha^{2-2p}$.
\end{proposition}
\begin{proof}
  The chain rule (resp. Lemma~\ref{lem:canform}) implies for the left
  hand side of \eqref{eq:2.2.1}
  \[\begin{aligned}
    F(u)(x)=& \frac{1}{\alpha}\nu\tr\bigl(\Hess(v) M^{-1}
    M^{-\top}\bigr) \bigl(M^{-1}(x-b)\bigr)\\
    &- \frac{1}{\alpha^p}
    \bigl|v(M^{-1}(x-b))\bigr|^{p-1} \sgn\bigl(v(M^{-1}(x-b))\bigr)
    Dv\bigl(M^{-1}(x-b)\bigr)M^{-1}\e_1
  \end{aligned}
  \]
  for all $x\in\R^d$ (respectively for almost every $x\in\R^d$ in the
  $H^2$-case).  Similarly, the right hand side of \eqref{eq:2.2.1} equals
  for all $x\in\R^d$ (resp. for a.e.\ $x\in\R^d$)
  \[\const_{(\alpha,M,b)} \frac{1}{\alpha}\bigl(\nu\tr\bigl(\Hess
  (v)\bigr)+|v|^{p-1}\sgn(v) \frac{\partial}{\partial x_1}v)\bigl(
  M^{-1} (x-b)\bigr).\]
  Therefore, \eqref{eq:2.2.1} holds as a pointwise equality for all $x\in\R^d$
  (resp. for a.e.\ $x\in\R^d$) if and only
  if $M^{-1}M^{-\top}=\const_{(\alpha,M,b)} I$ and
  $M^{-1}\e_1=\alpha^{p-1}\const_{(\alpha,M,b)}\e_1$.  These last two
  equalities are equivalent to $\const_{(\alpha,M,b)}
  =\alpha^{2-2p}$ and $M=\alpha^{p-1}\begin{pmatrix} 1&0\\0&Q
  \end{pmatrix}$ with $Q\in\O(d-1)$.
  Finally note that for $u\in H^2(\R^d)\cap L^\infty(\R^d)$ both sides of
  \eqref{eq:2.2.1} belong to $L^2(\R^d)$ and since they are equal for a.e.\
  $x\in \R^d$, the equality is an equality in $L^2$.
\end{proof}
Therefore, we consider the group
$G:=\bigl(\R_+\times\SO(d-1)\bigr)\ltimes\R^d$ whose elements we
denote by $g=(\alpha,Q,b)$, with components $\alpha\in(0,\infty)$,
$Q\in\SO(d-1)$, $b\in\R^d$.  The multiplication of two elements $(\alpha_1,
Q_1,b_1),(\alpha_2,Q_2,b_2)\in G$ is given by
\[(\alpha_1,Q_1,b_1)\bullet(\alpha_2,Q_2,b_2):=
\Bigl(\alpha_1\alpha_2,Q_1Q_2,b_1+
\alpha_1\left(\begin{matrix}1&0\\0&Q_1\end{matrix}\right) b_2\Bigr).\]
The group $G$ is a
non-compact, path connected Lie-group with unity element
$\mathbbm{1}=(1,I_{d-1},0)$ and the inverse of an element $(\alpha,Q,b)\in G$
is given by
$(\alpha,Q,b)^{-1}=\left(\tfrac{1}{\alpha},Q^\top,-\tfrac{1}{\alpha}
\left(\begin{smallmatrix}1&0\\0&Q^\top
\end{smallmatrix}\right) b\right)$.
Moreover, its Lie-algebra (tangent space at $\mathbbm{1}$) is
\[\liealg{g}=\T_\mathbbm{1}G=\R\times\so(d-1)\times\R^d.\]
For later use we denote the left-multiplication by some element
$g_0=(\alpha_0,Q_0,b_0)\in G$, by
\[L_{g_0}:G\to G,\;g\mapsto g_0\bullet g,\]
and the derivative (tangential) of $L_{g_0}$ at a point
$g_1=(\alpha_1,Q_1,b_1)\in G$ is
\begin{equation}\label{eq:GdLmult}
  T_{g_1}L_{g_0}: \T_{g_1}G\to \T_{g_0\bullet g_1}G,\quad
  \mu\mapsto T_{g_1}L_{g_0}\mu =\left(\alpha_0
  \mu_1,Q_0\mu_2,\alpha_0
  \left(\begin{matrix}1&0\\0&Q_0\end{matrix}\right)
  \mu_3\right).
\end{equation}
From \eqref{eq:GdLmult} we easily obtain the exponential map for $G$, i.e. the
mapping $\exp:\mu\in\liealg{g}\to g(1)\in G$, where $g$ is the solution to the
differential equation $g'(t)=T_\mathbbm{1}L_{g(t)}\mu,\, g(0)=\mathbbm{1}$.
For $\mu=(\mu_1,\mu_2,\mu_3)=(a,S,v)\in\liealg{g}$ we have the formula
\begin{equation}\label{eq:lieexp}
  \exp(\mu)= \Bigl(\exp(a),\exp(S),
  \sum_{k=0}^\infty \frac{1}{(k+1)!}
  \vect{a&0\\0&aI_{d-1}+S}^k v\Bigr),
\end{equation}
where $\exp(a)$ is the standard exponential in $\R$ and $\exp(S)$ is the
standard matrix exponential.  Note that by associativity
$L_{g_0}^{-1}=L_{{g_0}^{-1}}$ and therefore we have the identity
\begin{equation}\label{eq:GdLmult_inv}
  \left(T_{g_0}L_{{g_0}^{-1}}\right)^{-1}=T_{\mathbbm{1}}L_{g_0}.
\end{equation}

\begin{remark}\label{rem:2.3.1}
  The Lie group $G$ can be written as a matrix Lie group.  More
  precisely,
  \begin{equation}\label{eq:2.3.1}
    \mathcal{M}:G\to \R^{d+1,d+1},
    \quad \mathcal{M}:g=(\alpha,Q,b)\mapsto \mathcal{M}_g=
    \begin{pmatrix}\alpha&0&b_1\\
      0&\alpha Q&b_{2:d}\\
      0&0&1
    \end{pmatrix},
  \end{equation}
  where as before $\alpha>0$, $Q\in\SO(d-1)$,
  $b=(b_1,b_2,\dots,b_d)^\top\in\R^d$ and we abbreviate
  $b_{2:d}=(b_2,\dots,b_d)^\top$.  For
  $g=(\alpha,Q,b),g'=(\alpha',Q',b')\in G$ holds
  \[
  \mathcal{M}_g\mathcal{M}_{g'}=
  \begin{pmatrix} \alpha\alpha'&0&\alpha b_1'+b_1\\
    0&\alpha\alpha'QQ'&\alpha Q b_{2:d}'+b_{2:d}\\
    0&0&1
  \end{pmatrix}
  =\mathcal{M}_{g\bullet g'},\]
  and obviously $\mathcal{M}_{\mathbbm{1}}=I_{d+1}\in\R^{d+1,d+1}$, so that 
  $\mathcal{M}$ is a group homomorphism from $(G,\bullet)$ to the matrix Lie
  group (closed subgroup of $\GL(d+1;\R)$)
  \[\mathcal{M}_G:=
    \left\{M=\begin{pmatrix}\alpha&0&b_1\\0&\alpha Q&
      b_{2:d}\\0&0&1
    \end{pmatrix}\in\R^{d+1,d+1} :
    \alpha\in(0,\infty),\begin{pmatrix}b_1\\b_{2:d}
    \end{pmatrix}
    \in\R^d, 
  Q\in \SO(d-1)\right\}.\]
  The corresponding matrix Lie algebra of $\mathcal{M}_G$ is
  \begin{equation}
    \label{eq:2.3.2}
    \liealg{m}_G=
    \T_{I_{d+1}}\mathcal{M}_G=\left\{
      m=\begin{pmatrix}a&0&v_1\\0&aI_{d-1}+S&v_{2:d}\\0&0&0
      \end{pmatrix}: a\in\R, \begin{pmatrix}v_1\\v_{2:d}
      \end{pmatrix}
      \in\R^{d},
      S\in\so(d-1)\right\}.
  \end{equation}
  In this matrix setting the derivative of the left multiplication in
  $\mathcal{M}_G$ by some element $\mathcal{M}_{g_0}$ with
  $g_0=(\alpha_0,Q_0,b_0)\in G$ of course is
  the left multiplication by a matrix:
  \[L_{\mathcal{M}_{g_0}}:\mathcal{M}_G\to \mathcal{M}_G,\quad
  \mathcal{M}_g\mapsto
  L_{\mathcal{M}_{g_0}}\mathcal{M}_g=\mathcal{M}_{g_0}\mathcal{M}_g
  =\mathcal{M}_{g_0\bullet g}.\]
  The tangential of $L_{\mathcal{M}_{g_0}}$ at the identity is easily calculated
  to be
  \[ T_IL_{\mathcal{M}_{g_0}}:
    \begin{aligned}
      \T_I\mathcal{M}_G&\to \T_{\mathcal{M}_{g_0}} \mathcal{M}_G,\\
      m=
      \begin{pmatrix}a&0&v_1\\0&a I+S&v_{2:d}\\0&0&0\end{pmatrix}
      &\mapsto \mathcal{M}_{g_0}m = 
      \begin{pmatrix}\alpha a&0&\alpha v_1\\
        0&\alpha Q a+ \alpha Q S&\alpha Q v_{2:d}\\0&0&0
      \end{pmatrix}.
  \end{aligned}\]
  Finally, in this case the exponential map in $\mathcal{M}_G$
  is simply given by the matrix exponential, and one finds
  for $m=\vect{a&0&v_1\\0&aI+S&v_{2:d}\\0&0&0}\in\liealg{m}_G$ the formula
  \begin{equation}\label{eq:2.3matrixexp}
    \exp(m) = \begin{pmatrix} 
      \exp(a)&0&\sum_{k=0}^\infty \tfrac{1}{(k+1)!}a^k v_1\\
      0&\exp(a)\exp(S)&\sum_{k=0}^\infty \tfrac{1}{k+1}
      (aI+S)^{k}v_{2:d}\\ 
      0&0&1
    \end{pmatrix}.
  \end{equation}
  Note that comparing the entries in \eqref{eq:2.3matrixexp} we again obtain
  formula \eqref{eq:lieexp} for the exponential map in $G$.
\end{remark}


%% file: sec3.tex
\section{The Action and its Generators}
\label{sec:act_gen}
In the previous section we derived a continuous family of symmetries (Lie group)
which basically commutes with the nonlinear vector field $F$ from
\eqref{eq:burgersdd'}, see \eqref{eq:2.2.1}.  In this section we now consider the
analytic properties of how these symmetries act on functions.  We also calculate
the generators of these actions.
From now on let the parameter $p>1$ in the vector field $F(u)=\nu\laplace
u+\tfrac{1}{p}\tfrac{\partial}{\partial x_1}\bigl(|u|^p\bigr)$ be fixed.  We
denote by $a$ the action of the symmetry group $G$ on
functions $v:\R^d\to\R$,
\begin{equation} \label{eq:defa}
  a(\alpha,Q,b)v(x):=
  a\bigl((\alpha,Q,b),v\bigr)(x):=
  \alpha^{-1}
  v\left(\alpha^{1-p}\wt{Q}^\top(x-b)\right)
  \quad\forall
  x\in\R^d,
\end{equation}
where $\wt{Q}$ denotes the augmented matrix
$\wt{Q}=\left(\begin{smallmatrix}1&0\\0&Q\end{smallmatrix}\right)$ for
brevity.  Obviously, $a$ is a linear left action.  
Furthermore, we denote by $m$ the group homomorphism 
\begin{equation}\label{eq:defm}
  m:G\to (\R_+,\cdot),\;
  m(\alpha,Q,b)=\alpha^{2-2p}.
\end{equation}

First we show that the group action $a$ of $G$ 
is indeed a strongly continuous group action on $\mathcal{C}_0^k(\R^d)$ and
$H^k(\R^d)$.
\begin{lemma}\label{lem:2.3.1}
  Let $X=\mathcal{C}_0^k(\R^d)$ or $X=H^k(\R^d)$.
  Let $p>1$ be fixed and the action $a$ of $G$ on functions be given by
  \eqref{eq:defa}.
  Then $a:G\times X\to X, (g,v)\mapsto a(g)v$ is a strongly
  continuous group action, i.e.
  \begin{enumerate}[label=\textup{(\roman*)}]
  \item \label{itm:lem:2.3.1a} $a(g)\in \GL(X)\quad\forall g\in G$,
  \item \label{itm:lem:2.3.1b} $a(\mathbbm{1})=\operatorname{id}_{X}$,
  \item \label{itm:lem:2.3.1c} $a(g\bullet h)=a(g)\circ a(h)\quad \forall
    g,h\in G$,
  \item \label{itm:lem:2.3.1d} $\lim_{g\to\mathbbm{1}}a(g)v=v\quad\forall
    v\in X$.
  \end{enumerate}
\end{lemma}
Using the Banach-Steinhaus Theorem, Lemma~\ref{lem:2.3.1} easily
implies that $a$ is continuous:
\begin{corollary}
  The mapping $a:G\times X\to X, (g,v)\mapsto a(g) v$ is
  continuous.
\end{corollary}
\begin{proof}[Proof of Lemma~\ref{lem:2.3.1}.]
  Throughout the proof we always use $g=(\alpha,Q,b)\in G$.
  We begin with some preliminaries.  For $v\in \mathcal{C}^0_0(\R^d)$
  holds
  \begin{equation}
    \label{eq:2.3.3}
    \|a(g) v\|_\infty = \frac{1}{\alpha} \|v\|_\infty,
  \end{equation}
  and for $v\in \mathcal{S}(\R^d)$ we obtain by using the transformation
  formula the equality
  \begin{equation}
    \label{eq:2.3.4}
    \|a(g) v\|_{L^2}^2 = \int_{\R^d} \frac{1}{\alpha^2} \left|
      v\left(\frac{1}{\alpha^p} \wt{Q}^\top(x-b)\right)\right|^2\,dx
    = \alpha^{pd-2} \|v\|_{L^2}^2,
  \end{equation}
  which extends to all of $L^2(\R^d)$ by density of $\mathcal{S}(\R^d)$ in
  $L^2(\R^d)$.

  Now let $v$ be sufficiently smooth, e.g. $v\in\mathcal{S}(\R^d)$,
  which is dense in both $\mathcal{C}_0^k(\R^d)$ and $H^k(\R^d))$ 
  (e.g.~\cite[\S~2.6]{Rauch:1991}).
  For $i\in 1,\ldots,d$ then follow with the chain rule
  \begin{equation*}
    \frac{\partial}{\partial x_i} [ a(\alpha,Q,b) v](x) =
    \frac{1}{\alpha} Dv\left(\frac{1}{\alpha^p} \wt{Q}^\top
      (x-b)\right) \frac{1}{\alpha^p} \wt{Q}^\top e_i 
    = \frac{1}{\alpha^p} [a(g) v_i](x),
  \end{equation*}
  where we denote $v_i(x) = Dv(x) \wt{Q}^\top e_i$, so that inductively for
  every multi-index $\beta\in \N^d$ holds
  \begin{equation}\label{eq:2.3.5}
    D^\beta [a(\alpha,Q,b) v] = \frac{1}{\alpha^{p|\beta|}}
    \left[ a(g)\Bigl( D^{|\beta|}v(\cdot)[\wt{Q}^\top
      e^\beta]\Bigr)\right].
  \end{equation}
  Here we abbreviate $[\wt{Q}^\top e^\beta] = [\wt{Q}^\top
  e_1,\ldots,\wt{Q}^\top e_1,\ldots,\wt{Q}^\top e_d]$, 
  where each $\wt{Q}^\top e_i$ is repeated $\beta_i$-times,
  $|\beta|=\sum_{i=1}^d \beta_i$ as usual for multi-indices, and $D^{|\beta|}v$
  is the $\beta$'th total derivative of $v$.
  Because $\wt{Q}$ is an orthogonal matrix and $|e_i|=1$ for all $i=1,\dots,d$,
  the absolute value of the function
  $D^{|\beta|}v(\cdot)(\wt{Q}^\top e^\beta)$ is pointwise bounded by
  \begin{equation}\label{eq:2.3.6}
    \left|D^{|\beta|}v(x)(\wt{Q}^\top e^\beta)\right|
    \leq \const(|\beta|) \sum_{\nu\leq |\beta|} |D^\nu v(x)|,
  \end{equation}
  where $\const(|\beta|)$ only depends on $|\beta|$.
  Now we are ready to prove \ref{itm:lem:2.3.1a}.  Linearity and invertibility
  are obvious so that it remains to show boundedness in the respective spaces.
  First let $v\in \mathcal{C}^k_0(\R^d)$.  Then for $g=(\alpha,Q,b)\in G$ follows
  \begin{multline*}
    \|a(g)v\|_{\mathcal{C}^k} = \max_{|\beta|\leq k}
    \frac{1}{\alpha^{p|\beta|}} \|a(g)\bigl(D^{|\beta|}
    v(\cdot)[\wt{Q}^\top e^\beta]\bigr)\|_\infty\\
    \leq \left(1+\frac{1}{\alpha^{pk}}\right) 
    \max_{|\beta|\leq k} \sup_{x}\left| \frac{1}{\alpha} D^{|\beta|}
      v\bigl(\frac{1}{\alpha^p} \wt{Q}^\top(x-b)\bigr)[\wt{Q}^\top
      e^\beta]\right|\\
    \leq  \left(1+\frac{1}{\alpha^{pk}}\right) 
    \frac{1}{\alpha} \const(k) \|v\|_{\mathcal{C}^k},
  \end{multline*}
  where we used~\eqref{eq:2.3.5}, \eqref{eq:2.3.6}, and \eqref{eq:2.3.3}.  This implies
  \begin{equation}\label{eq:2.3.7}
    \|a(g)\|_{\GL(\mathcal{C}^k_0)}\leq \const(\alpha,k),
  \end{equation}
  where $\const(\alpha,k)$ can be chosen independently of $\alpha$ for
  all $\alpha$ from a fixed compact subset of $(0,\infty)$ and hence for all $g$
  from a fixed compact subset of $G$.

  Now consider the boundedness in $H^k(\R^d)$.  Let $v\in \mathcal{S}(\R^d)$
  and recall that the Fourier transform maps $\mathcal{S}(\R^d)$ into
  $\mathcal{S}(\R^d)$ (e.g.~\cite[Thm.~7.4]{Rudin:1991}).  We first
  observe
  \begin{equation}\label{eq:2.3.8}
    \bigl(a(g)v\bigr)^\wedge(\xi)=
    (2\pi)^{-\frac{d}{2}} \int_{\R^d} e^{-ix^\top \xi}
    \frac{1}{\alpha} v\left( \frac{1}{\alpha^p} \wt{Q}^\top
      (x-b)\right)\,dx =
    e^{-ib^\top \xi} \alpha^{dp-1} \wh{v}(\alpha^p \wt{Q}^\top \xi),
  \end{equation}
  so that we find for the $\|\cdot\|_{\dot{H}^k}$-norm (see
  \eqref{eq:hkdotnorm})
  \begin{multline*}
    \|a(g) v\|_{\dot{H}^k} = \alpha^{2(dp-1)} \int_{\R^d} 
    \langle\xi\rangle^{2k} \bigl| \wh{v}(\alpha^p \wt{Q}^\top
    \xi)\bigr|^2 \, d\xi
    =\alpha^{dp-2}\int_{\R^d} \langle \wt{Q}\alpha^{-p} z\rangle^{2k}
    |\wh{v}(z)|^2\,dz\\
    \leq \alpha^{dp-2}(1+\alpha^{-2pk}) \|v\|_{\dot{H}^k}^2
    =\const(\alpha,d,k) \|v\|_{\dot{H}^k}^2.
  \end{multline*}
  Again the constant $\const(\alpha,d,k)$ is uniformly bounded in
  compact subsets of $G$.  Because $\mathcal{S}(\R^d)$ is dense in
  $H^k(\R^d)$ and the norm $\|\cdot\|_{\dot{H}^k}$ is equivalent to the
  $\|\cdot\|_{H^k}$-norm, the same estimate holds for any $v\in H^k$ and 
  \begin{equation}\label{eq:2.3.9}
    \|a(g)\|_{\GL(H^k)}\leq \const(\alpha,k,d)
  \end{equation}
  follows.  This finishes the proof of
  \ref{itm:lem:2.3.1a}.
  
  By definition of the action $a$ also \ref{itm:lem:2.3.1b} and
  \ref{itm:lem:2.3.1c} hold.

  For the proof of~\ref{itm:lem:2.3.1d} we again first consider the
  $\mathcal{C}^k_0$-case:\\
  Let $v\in \mathcal{C}^0_0$.  For $g=(\alpha,Q,b)\in G$ we then
  obtain the estimate
  \begin{equation}\label{eq:ck0est1}
    \begin{aligned}
      \|a(g) v-v\|_\infty &= \sup_{x\in \R^d} \left| \alpha^{-1}
        v\bigl(\alpha^{-p} \wt{Q}^\top(x-b)\bigr) -
        v(x)\right|\\
      &\leq \alpha^{-1} \sup_{|x|\geq R} \left| v\bigl(
        \alpha^{-p} \wt{Q}^\top(x-b)\bigr)\right|
      + \sup_{|x|\geq R}
      |v(x)|\\
      &\quad + \alpha^{-1} \sup_{|x|\leq R} \left| v\bigl(
        \alpha^{-p} \wt{Q}^\top(x-b)\bigr)-v(x)\right|
       + \left|\alpha^{-1}-1\right| \sup_{|x|\leq R}
      |v(x)|. 
    \end{aligned}
  \end{equation}
  To show that $\|a(g)v-v\|_\infty\to 0$ as $g\to\mathbbm{1}$, let
  $\epsilon>0$ be given.  Since $\lim_{|x|\to\infty} |v(x)|=0$, there
  is $R>0$, so that the first two summands on the right hand side of
  \eqref{eq:ck0est1} are smaller than $\epsilon$ for all $|b|<1$ and
  $|1-\alpha|<\frac{1}{2}$.  Because $v$ is uniformly continuous on
  compact subsets and $\alpha^{-p}\wt{Q}^\top(x-b)\to x$ uniformly on
  $\{|x|\leq R\}$ as $\alpha\to 1, Q\to I, b\to 0$, the third summand
  converges to zero as $g\to \mathbbm{1}$.  Also the last summand converges to
  zero as $g\to\mathbbm{1}$, since $\alpha\to 1$ and $v$ is uniformly bounded.
  
  Now let $v\in \mathcal{C}^k_0$ and $\beta\in\N^d$ be some multi-index with $0<|\beta|\leq k$.
  Then we use \eqref{eq:2.3.5} to estimate for $g=(\alpha,Q,b)\in G$
  \begin{equation}\label{eq:ck0est2}
    \begin{aligned}
      \left\|D^\beta ([a(g)v]-v)\right\|_\infty
      &\leq \left| 1-\alpha^{-p|\beta|}\right|\,
      \|a(g)\|_{\GL(\mathcal{C}_0^0)}\,
      \|D^{|\beta|}v(\cdot)[\wt{Q}^\top e^\beta]\|_\infty\\
      &\quad  + \|a(g)\|_{\GL(\mathcal{C}^0_0)}\, \|D^{|\beta|}
      v(\cdot)[\wt{Q}^\top
      e^\beta]-D^{|\beta|}v(\cdot)[Ie^\beta]\|_\infty \\
      &\quad + \|a(g) D^{\beta}v(\cdot)-D^\beta v(\cdot)\|_\infty.
    \end{aligned}
  \end{equation}
  The first summand on the right hand side of \eqref{eq:ck0est2} converges to
  zero as $g\to\mathbbm{1}$, because $\|a(g)\|_{\GL(\mathcal{C}^0_0)}$ is uniformly bounded for
  $g$ from a compact subset of $G$ by \eqref{eq:2.3.7} and the factor
  $\|D^{|\beta|}v(\cdot)[\wt{Q}^\top e^\beta]\|_\infty$ is uniformly bounded
  by \eqref{eq:2.3.6}.
  The second summand converges to zero as $g\to\mathbbm{1}$ again because of
  \eqref{eq:2.3.7} and since $[\wt{Q}^\top e^\beta]\to [I\,e^\beta]$ in
  $(\R^d)^{|\beta|}$ (recall the abbreviation introduced above) 
  and the $|\beta|$-multilinear forms $D^{|\beta|}v(x)$ are bounded
  independently of $x$.
  Finally, also the third summand converges to zero as $g\to\mathbbm{1}$
  because of the $\mathcal{C}_0^0$-case, considered above.
  This proves \ref{itm:lem:2.3.1d} for the $\mathcal{C}^k_0(\R^d)$-case.

  For the $H^k$-case it suffices to consider $v\in \mathcal{S}(\R^d)$
  because of estimate~\eqref{eq:2.3.9} and the density of
  $\mathcal{S}(\R^d)$ in $H^k(\R^d)$.  Therefore, let $v\in
  \mathcal{S}(\R^d)$ and let $(g_n)_{n\in\N}\subset G$ be a sequence with
  $g_n\to\mathbbm{1}$ as $n\to\infty$.
  Using~\eqref{eq:2.3.8} we find
  \[
  \|a(g_n)v-v\|_{\dot{H}^k}^2 =\int_{\R^d} \langle\xi\rangle^{2k}
  \Bigl| \alpha_n^{dp-1} \wh{v}(\alpha_n^p\wt{Q}_n^\top\xi)
  e^{-ib_n^\top \xi} - \wh{v}(\xi)\Bigr|^2\,d\xi.\]
  The integrand converges pointwise to zero as $n\to\infty$ and
  is bounded by
  \begin{equation}\label{eq:hkest1}
    2\langle\xi\rangle^{2k}  \alpha_n^{2dp-2}
    \wh{v}(\alpha_n^p\wt{Q}_n^\top\xi)^2 + 
    2\langle\xi\rangle^{2k}  \wh{v}(\xi)^2.
  \end{equation}
  For $n_0$ with $|\alpha_n-1|\le \tfrac{1}{2}$ for all $n\ge n_0$, both
  summands of \eqref{eq:hkest1} are uniformly integrable independently of $n\ge
  n_0$.  Finally, for any given $\epsilon>0$ there exists $R>0$,
  such that for all $n\geq n_0$ holds
  \[
  \int_{|x|>R} 2 \langle\xi\rangle^{2k} \Bigl( \alpha_n^{2dr-2}
  \wh{v}(\alpha_n^p\wt{Q}_n^\top\xi)^2
  + \wh{v}(\xi)^2\Bigr)\,d\xi\leq \epsilon.\]
  Thus, the Vitali convergence theorem
  (e.g.~\cite[VI.5.6]{Elstrodt:2009}) applies and yields
  \[\lim_{n\to\infty}\|a(g_n)v-v\|_{\dot{H}^k}=0.\]
\end{proof} 

With the above definitions of $G$, $a$ and $m$, and because of
Lemma~\ref{lem:2.3.1}, we can rephrase
Proposition~\ref{prop:symmetries} as follows:
\begin{proposition}\label{prop:symburg}
  For all $v\in \mathcal{C}^2(\R^d)$ (resp. $v\in H^2(\R^d)\cap
  L^\infty(\R^d)$)
  hold for all $g\in G$ 
  \begin{equation}\label{eq:rel}
    F\bigl(a(g)v\bigr)=m(g) a(g)F(v),
  \end{equation}
  as identities in $\mathcal{C}^0$ (resp. in $L^2$).
\end{proposition}
As before, let $X=\mathcal{C}^k_0(\R^d)$ or $X=H^k(\R^d)$.
For fixed $v\in X$ we denote by $T_ga\,v:\T_gG\to \T_{a(g)v}X$ the tangential of the
mapping $G\ni g\mapsto a(g)v\in X$ at $g\in G$, if it exists.  The
evaluation of $T_ga\,v$ at $\mu\in \T_gG$ is denoted by
$T_ga\,v[\mu]$.  
To motivate that the set of all $v\in X$ for which the tangential exists is
actually a useful set, note that by Lemma~\ref{lem:2.3.1} and the properties of
the exponential map for each $\mu\in\liealg{g}$ the mapping $S_\mu:t\mapsto
a\bigl(\exp(\mu t)\bigr)$ defines a strongly continuous semigroup on $X$.  Hence
the generator is a densely defined and closed operator on $X$.  The generator
coincides with the directional derivative of the group action at $\mathbbm{1}$
in the direction $\mu$.
This motivates to consider for a basis
$\epsilon_1,\ldots,\epsilon_{\dim(\liealg{g})}$ of
$\liealg{g}$ the subset $Y_0$ of $X$ defined as
\begin{equation}\label{eq:2.3.10}
  Y_0:=\left\{ v\in X:\exists T_{\mathbbm{1}}a\,v[\epsilon_j] 
    = 
    \lim_{t\to 0} \frac{1}{t} \bigl(a(\exp(\epsilon_j t))v-v\bigr),\;
    j=1,\ldots,\dim(\liealg{g})\right\}.
\end{equation}
Using Lemma~\ref{lem:2.3.1} and the results of
\cite[\S~4]{SandstedeScheelWulff:1999}, we obtain that $Y_0$ indeed is a
dense subset of $X$, and $Y_0$ becomes a Banach-space, when endowed
with the norm
\[\|v\|_{Y_0}=\|v\|_{X}+\sup_{j=1,\ldots,\dim\liealg{g}} \|T_{\mathbbm{1}}
a\,v[\epsilon_j]\|_X.\]
Moreover, $G$ acts continuously on $\bigl(Y_0,\|\cdot\|_{Y_0}\bigr)$ and for fixed $v\in Y_0$ the
mapping 
\[G\ni g\mapsto a(g)v\in X\]
is continuously differentiable. 
We summarize the above discussion in the following lemma.
\begin{lemma}\label{lem:op_prop_action}
  Let $X=\mathcal{C}_0^k(\R^d)$ or $X=H^k(\R^d)$ and let $Y_0$ be
  defined by~\eqref{eq:2.3.10}.  Then $Y_0\subset X$ is a dense
  subset and the action $a$, defined
  in~\eqref{eq:defa} has the following properties:
  \begin{itemize}
  \item[(a)] For fixed $g\in G$, the operator $a(g)$ is a
    bounded linear operator on $X$ and on $Y_0$.
  \item[(b)] For fixed $v\in X$, resp.\ $v\in Y_0$, the mapping $G\ni
    g\mapsto a(g) v$ is continuous into $X$, resp.\ $Y_0$.
  \item[(c)] For fixed $v\in Y_0$, the mapping $G\ni g\mapsto a(g)v\in
    X$, is continuously differentiable.
  \end{itemize}
\end{lemma}
\begin{example}\label{ex:gen}
  For later use, we now explicitly calculate for $d=1$, $d=2$, and $d=3$ spatial
  dimensions for specific choices of
  bases $\epsilon_1,\dots,\epsilon_{\dim(\liealg{g})}$ of $\liealg{g}$ the
  generators $v\mapsto T_\mathbbm{1} a\,v[\epsilon_j]$ of the semigroups
  $\bigl(S_{\epsilon_j}(t)\bigr)_{t\ge 0}$ and the space $Y_0$.  As before, we
  may restrict the calculation to $v\in \mathcal{S}(\R^d)$, which is dense in
  both, $\mathcal{C}_0^k(\R^d)$ and $H^k(\R^d)$.
  \begin{enumerate}[label=\textup{(\roman*)}]
    \item \label{itm:ex:gen1d} For $d=1$ the symmetry group is
      $G=\bigl(\R_+\times\emptyset\bigr)\ltimes \R\hat{=} \R_+\ltimes \R$.
      As a basis $\{\epsilon_1,\epsilon_2\}$ of
      $\liealg{g}=T_\mathbbm{1}G=\R\times\R$ we choose
      \begin{equation*}
        \epsilon_1=(1,0), \epsilon_2=(0,1).
      \end{equation*}
      From formula \eqref{eq:lieexp} for the exponential map we obtain
      $\exp(\epsilon_1 t) = (e^t,0)$ and $\exp(\epsilon_2 t) = (0,t)$, so that
      together with \eqref{eq:defa} follow
        \begin{align*}
          T_\mathbbm{1}a\,v[\epsilon_1]&= \lim_{t\searrow 0} \frac{e^{-t}
          v(e^{(1-p)t} \cdot)-v(\cdot)}{t}= -v+(1-p)xv_x,\\
          T_\mathbbm{1}a\,v[\epsilon_2]&= \lim_{t\searrow 0}
          \frac{v(\cdot-t)-v(\cdot)}{t} = -v_x.
        \end{align*}
      Therefore, $Y_0=\{v\in X:v_x\in X, xv_x\in X\}$.
    \item \label{itm:ex:gen2d} For $d=2$ we obtain
      $G=\bigl(\R_+\times\{1\}\bigr)\ltimes\R^2$.  As basis of
      $\liealg{g}=\T_\mathbbm{1}G=\R\times\{0\}\times\R^2$
      we choose
      \begin{equation*}
        \epsilon_1=(1,0,0), \epsilon_2=(0,0,e_1), \epsilon_3=(0,0,e_2).
      \end{equation*}
      Similar considerations as in the 1-dimensional case lead to
        \begin{align*}
          T_\mathbbm{1}a\,v[\epsilon_1]&=-v+(1-p)x^\top\gradient v,\\
          T_\mathbbm{1}a\,v[\epsilon_2]&=-v_{x},\\
          T_\mathbbm{1}a\,v[\epsilon_3]&=-v_{y},
        \end{align*}
      and $Y_0=\{v\in X: v_x, v_y, xv_x+yv_y \in X\}$.
    \item \label{itm:ex:gen3d} For $d=3$ we have $G=\bigl(\R_+\times
      \SO(2)\bigr)\ltimes\R^3$ and choose the basis 
      \begin{equation*}
        \epsilon_1=(1,0,0),
        \epsilon_2=\Bigl(0,\left(\begin{matrix}0&-1\\
          1&0\end{matrix}\right),0\Bigr),
        \epsilon_j=(0,0,e_{j-2}), j=3,4,5
      \end{equation*}
      of
      $\liealg{g}=T_\mathbbm{1}G=\R\times\linearhull\{\left(\begin{smallmatrix}0&-1\\
        1&0\end{smallmatrix}\right)\}\times\R^3$.  
      Analogous calculations as in the $d=1$ and $d=2$ case now yield the
      generators
        \begin{align*}
          T_\mathbbm{1}a\,v[\epsilon_1]&=-v+(1-p)x^\top\gradient v,\\
          T_\mathbbm{1}a\,v[\epsilon_2]&=x^\top
          \left(\begin{smallmatrix}0&0&0\\0&0&-1\\0&1&0
          \end{smallmatrix}\right) \gradient v,\\
          T_\mathbbm{1}a\,v[\epsilon_j]&=-v_{x_{j-2}},\;j=3,4,5
        \end{align*}
      and $Y_0=\{v\in X:v_x,v_y,v_z,xv_x+yv_y+zv_z, zv_y-yv_z\in X\}$, where
      we interchangeably write $v_x=v_{x_1}$, $v_y=v_{x_2}$, and $v_z=v_{x_3}$.
  \end{enumerate}
\end{example}

In the following we will  often use the spaces $\CC_\scal^2(\R^d)$ and
$H^2_\scal(\R^d)$, which are defined as
\begin{equation}\label{eq:defCCscal}
  \mathcal{C}^2_\scal(\R^d):=\Bigl\{v\in \mathcal{C}^2_0(\R^d): 
  T_\mathbbm{1}a\,v[\mu]\in \CC^0_0(\R^d)\;\forall
  \mu\in\liealg{g}\Bigr\}
\end{equation}
and
\begin{equation}\label{eq:defHscal}
  H^2_\scal(\R^d):=\Bigl\{v\in H^2(\R^d): 
  T_\mathbbm{1}a\,v[\mu]\in L^2(\R^d)\;\forall
  \mu\in\liealg{g}\Bigr\}.
\end{equation}
Endowed with a suitable norm, these become Banach-spaces
(respectively Hilbert-spaces).  
For concreteness we choose in the case of one spatial dimension the norms
$\|v\|_{\CC^2_\scal}=\|v\|_{\CC^2}+\|xv_x\|_{\infty}$, respectively
$\|v\|_{H^2_\scal}^2=\|v\|_{H^2}^2+\|xv_x\|_{L^2}^2$.  In the $d=2$ case 
the
norms $\|v\|_{\CC^2_\scal}=\|v\|_{\CC^2}+\|x^\top\gradient
v\|_\infty$, respectively
$\|v\|_{H^2_\scal}^2=\|v\|_{H^2}^2+\|x^\top\gradient v\|_{L^2}^2$, are
suitable.  And in the $d=3$ case, we choose
$\|v\|_{\CC^2_\scal}=\|v\|_{\CC^2}+\|x^\top\gradient
v\|_\infty+\|-yv_z+zv_y\|_\infty$, respectively
$\|v\|_{H^2_\scal}^2=\|v\|_{H^2}^2+\|x^\top\gradient v\|_{L^2}^2+
\|-yv_z+zv_y\|_{L^2}^2$.

\begin{remark}\label{rem:2.5}
  It is easy to see that the generator of the scaling action, given by
  \[T_\mathbbm{1}a\,v[(1,0,0)]=-v+(1-p)x^\top \gradient
  v=-v+(1-p)\divergence (x v)+d(p-1)v\] 
  in any space dimension, is in divergence form for all
  $v\in \mathcal{S}(\R^d)$ if and only if
  \[1=d(p-1)\Leftrightarrow p=\frac{d+1}{d}.\] 
  In contrast to this, the generators of the translation and rotation actions
  are always in divergence form.
  This is precisely the case for the action of the symmetry belonging to, what
  we called, the ``conservative Burgers' equation'' in the
  Section~\ref{sec:intro}.   This observation is precisely the reason, why we
  called the case $p=\frac{d+1}{d}$ in \eqref{eq:burgersdd} the conservative
  Burgers' equation in the first place.
  Note that the ``conservative Burgers' equation''
  coincides with the ``multi-dimensional Burgers'
  equation'' only in the case $d=1$.  
\end{remark}


%% file: sec4.tex
\section{Freezing Similarity Solutions}
\label{sec:freeze}
In this section we derive the equations for the numerical method of freezing
similarity solutions in generalized Burgers' equation.  In the first part we
derive the equation \eqref{eq:burgersdd} in new, time-dependent, coordinates,
based on the symmetries obtained in Sections~\ref{sec:can_sym}
and~\ref{sec:act_gen}.  In the second part we augment this new system with algebraic
constraints to deal with the arbitrariness introduced by the new coordinates.
In the following we denote $X=\CC^0_0(\R^d)$ and then $Y_1=\CC^2_\scal(\R^d)$ or
$X=L^2(\R^d)$ and then $Y_1=H^2_\scal(\R^d)$,
respectively.

\subsection{Co-moving Frames and Similarity Solutions}
The basic idea of the method is to split the time evolution of the
solution of the Cauchy problem for \eqref{eq:burgersdd'}
into a part that treats the evolution which
mainly takes place in the group orbit of the profile and a part that
takes the evolution of the profile in the remaining directions into
account.  Formally, this is done by making the ansatz
\begin{equation}\label{eq:2.4.aa}
  u(t)=a\bigl(g(\sigma(t))\bigr) v\bigl(\sigma(t)\bigr),
\end{equation}
where $a$ is the left action of the symmetry group $G$, defined in
\eqref{eq:defa}, $g$ is a smooth curve in the group $G$,
$\sigma$ is an orientation preserving diffeomorphism of two intervals,
describing a scaling of time, and $v$ is a time dependent profile, taking values
in $Y_1$.

A solution of the Cauchy-problem, whose evolution solely takes
place in the group orbit, we call a similarity solution.  This is made precise
in the next definition.
\begin{definition}[Similarity solution]\label{def:simsol}
  We call $u$ a similarity solution of $u_t=F(u)$ with profile
  $\sol{v}\in Y_1$, if there exists an
  open interval $J\subset \R$, a differentiable map
  $\sigma:J\to\R$ with $\dot{\sigma}(t)>0$
  for all $t\in J$, and a differentiable curve $g\in
  \mathcal{C}^1(\sigma(J),G)$,
  such that $u(t)=a\bigl(g(\sigma(t))\bigr) \sol{v}$ is a
  solution of $u_t=F(u)$.
\end{definition}
The key for the freezing method is now the following
Theorem~\ref{thm:keyfreeze}, which relates the differential equations
satisfied by the two functions $u$ and $v$ from the ansatz \eqref{eq:2.4.aa}.
\begin{theorem}\label{thm:keyfreeze}
  Let $u_0\in Y_1$ and $\mu\in \CC([0,\wh{T});\liealg{g})$, $\wh{T}>0$
  be given.  Then there exist unique maximally extended solutions
  $g\in \CC^1([0,\wh{T});G)$ of
  \begin{subequations}\label{eq:recbasic}
    \begin{equation}\label{eq:thm.2.1}
      g'(\tau)=T_\mathbbm{1}L_{g(\tau)}\mu(\tau),\quad g(0)=\mathbbm{1},
    \end{equation}
    and $\sigma\in \CC^1([0,T);[0,\wh{T}))$ of
    \begin{equation}\label{eq:thm.2.2}
      \dot{\sigma}(t)=m\bigl(g(\sigma(t))\bigr),\quad \sigma(0)=0.
    \end{equation}
  \end{subequations}
  The function $\sigma:[0,T)\to[0,\wh{T})$ is a diffeomorphism.
  Furthermore, the following statements hold true:
  \begin{enumerate}[label=\textup{(\roman*)}]
  \item\label{itm:thm2.10.i} If $u\in \CC([0,T);Y_1)\cap
    \CC^1([0,T);X)$ solves the Cauchy problem for \eqref{eq:burgersdd'} with
      initial condition $u(0)=u_0\in Y_1$, then $v:\tau\mapsto
    a\bigl(g(\tau)^{-1}\bigr) u\bigl(\sigma^{-1}(\tau)\bigr)$ belongs to
    $\CC([0,T);Y_1)\cap\CC^1([0,T);X)$ and solves
    \begin{equation}\label{eq:2.4.1}
      v_\tau = F(v)-T_\mathbbm{1}a\,v[\mu],\quad v(0)=u_0.
    \end{equation}
  \item\label{itm:thm2.10.ii} If $v\in \CC([0,\wh{T});Y_1)\cap
    \CC^1([0,\wh{T});X)$ solves the Cauchy problem \eqref{eq:2.4.1}, then $u:t\mapsto
    a\bigl(g(\sigma(t))\bigr) v(\sigma(t))$ belongs to
    $\CC([0,T);Y_1)\cap \CC^1([0,T);X)$ and solves the Cauchy problem for 
    \eqref{eq:burgersdd'} with $u(0)=u_0$.
  \end{enumerate}
\end{theorem}
\begin{proof}
  By \eqref{eq:GdLmult}, the differential equation for the first
  component of $g$ in \eqref{eq:thm.2.1} decouples and can be solved
  first.  It follows that the solution for the first component exists
  globally in $[0,\wh{T})$.  In a subsequent step, the remaining
  differential equations for the other components of $g$ can be solved.  More
  precisely, knowing the first component of $g$ in $[0,\wh{T})$, the remaining
  ODEs become linear
  and hence the solution $g$ exists globally and belongs to
  $\CC^1([0,\wh{T});G)$.  In the next step, one uses that $g$ is known
  and $m\circ g\in \CC^1([0,\wh{T}),\R_+)$.  Therefore,
  \eqref{eq:thm.2.2} has a unique maximally extended solution
  $\sigma\in \CC^1([0,T);[0,\wh{T}))$.  Because of the differential
  equation $\dot{\sigma}(t)=m\bigl(g(\sigma(t))\bigr)>0$ for all
  $t\in[0,T)$ and $\sigma:[0,T)\to[0,\wh{T})$ is a diffeomorphism.

  Proof of \ref{itm:thm2.10.i}. 
  The smoothness of $v$ follows from Lemma~\ref{lem:op_prop_action}
  and the assumptions on $u$.  Because $a$ is a group homomorphism,
  $v(\tau)=a\bigl(g(\tau)^{-1}\bigr)u\bigl(\sigma^{-1}(\tau)\bigr)$ is
  equivalent to
  \begin{equation}\label{eq:2.4.1a}
    a\bigl(g(\tau)\bigr) v(\tau)=
    u\bigl(\sigma^{-1}(\tau)\bigr)\quad\forall \tau\in[0,\wh{T}).
  \end{equation}
  
  First consider the left hand side of \eqref{eq:2.4.1a}.  For
  $h\in\R$, $h$ small, with $\tau+h\in[0,\wh{T})$ we find as equalities in $X$
  \begin{equation*}
    \begin{aligned}
      &a\bigl(g(\tau+h)\bigr) v(\tau+h)-a\bigl(g(\tau)\bigr) v(\tau)\\
      &=a\bigl(g(\tau+h)\bigr) \bigl(v(\tau+h)-v(\tau)\bigr)+
      a\bigl(g(\tau+h)\bigr) v(\tau)
      -a\bigl(g(\tau)\bigr) v(\tau)\\
      &=a\bigl(g(\tau+h)\bigr) \bigl(v_\tau(\tau)h+o(|h|)\bigr)
      +a\bigl(g(\tau)\bigr)
      \Bigl(a\bigl(L_{g(\tau)^{-1}}g(\tau+h)\bigr) v(\tau) -
      a\bigl(L_{g(\tau)^{-1}}g(\tau)\bigr) v(\tau)\\
      &=a\bigl(g(\tau+h)\bigr)v_\tau(\tau)h+o(|h|) + 
      a\bigl(g(\tau)\bigr) T_\mathbbm{1}a\, v(\tau)\bigl[
      T_{g(\tau)}L_{g(\tau)^{-1}}\bigl(g'(\tau)h+o(|h|)\bigr)\bigr].
    \end{aligned}
  \end{equation*}
  For the last equality one must use a chart of $G$ at $g(\tau)$ and a
  chart of $G$ at $\mathbbm{1}$.  Note that the first $o$-term in the last line
  belongs to $X$ and the other one to $\T_{g(\tau)}G$.
  The above equalities show 
  \begin{multline}\label{eq:2.4.2a}
    \lim_{h\to 0} \frac{1}{h}\Bigl(
    a\bigl(g(\tau+h)\bigr) v(\tau+h)-a\bigl(g(\tau)\bigr)
    v(\tau)\Bigr)\\
    = a\bigl(g(\tau)\bigr) v_\tau(\tau) + a\bigl(g(\tau)\bigr)
    T_\mathbbm{1} a \, v(\tau)\bigl[T_{g(\tau)}L_{g(\tau)^{-1}}
    g'(\tau)\bigr],
  \end{multline}
  where the limit exists in $X$.  This is the derivative of the left
  hand side of \eqref{eq:2.4.1a} with respect to $\tau$.
  
  Differentiation of the right hand side of \eqref{eq:2.4.1a} at
  $\tau$ yields by the chain rule the identities
  \begin{equation}\label{eq:2.4.2b}
    u_t\bigl(\sigma^{-1}(\tau)\bigr)
    \frac{1}{\dot\sigma\bigl(\sigma^{-1}(\tau)\bigr)}
    =\frac{1}{m\bigl(g(\tau)\bigr)}
    F\bigl(u(\sigma^{-1}(\tau))\bigr)
    = m\bigl(g(\tau)^{-1}\bigr) F\bigl(u(\sigma^{-1}(\tau))\bigr),
  \end{equation}
  where we used \eqref{eq:thm.2.2}, \eqref{eq:burgersdd'}, and that $m$ is a
  group homomorphism.

  Equating \eqref{eq:2.4.2a} and \eqref{eq:2.4.2b} and
  application of $a\bigl(g(\tau)^{-1}\bigr)$ to both sides leads to
  \[v_\tau(\tau) + T_\mathbbm{1} a \,
  v(\tau)\bigl[T_{g(\tau)}L_{g(\tau)^{-1}}
  g'(\tau)\bigr]=a\bigl(g(\tau)^{-1}\bigr) m\bigl(g(\tau)^{-1}\bigr)
  F\bigl(u(\sigma^{-1}(\tau))\bigr)\]
  as an equality in $X$.  Using equation
  \eqref{eq:thm.2.1} and the symmetry property,
  Lemma~\ref{prop:symburg}, then proves the asserted equality
  \eqref{eq:2.4.1}.
  
  Proof of \ref{itm:thm2.10.ii}.  For all $t\in [0,T)$ let $u$ be given by
  \begin{equation}\label{eq:2.4.3}
    u(t)=a\bigl(g(\sigma(t))\bigr) v\bigl(\sigma(t)\bigr).
  \end{equation}
  The asserted smoothness of $u$ follows from
  Lemma~\ref{lem:op_prop_action} and we may differentiate
  \eqref{eq:2.4.3} with respect to $t$ as in the
  proof of \ref{itm:thm2.10.i}, which leads to
  \begin{equation}\label{eq:2.4.4}
    \begin{aligned}
    u_t(t)&=a\bigl(g(\sigma(t))\bigr)
    T_{\mathbbm{1}}a\,v\bigl(\sigma(t)\bigr)
    \bigl[T_{g(\sigma(t))}L_{g(\sigma(t))^{-1}}
    g'\bigl(\sigma(t)\bigr) \dot{\sigma}(t)\bigr] +
    a\bigl(g(\sigma(t))\bigr) v_\tau(\sigma(t))\dot{\sigma}(t)\\
    &=a\bigl(g(\sigma(t))\bigr)
    \Bigl(T_\mathbbm{1}a\,v\bigl(\sigma(t)\bigr)[\mu(\sigma(t))
    ] +v_\tau(\sigma(t))\Bigr)m\bigl(g(\sigma(t))\bigr)\\
    &=a\bigl(g(\sigma(t))\bigr)
    F\bigl(v(\sigma(t))\bigr)
    m\bigl(g(\sigma(t))\bigr)
    =F\bigl(u(t)\bigr),
    \end{aligned}
  \end{equation}
  where we used the differential equations for $g$ and $\sigma$ in the
  second equality, the PDE for $v$ in the third equality, and the
  symmetry property of $F$ (see Proposition~\ref{prop:symburg}) in the
  last equality.  Equation~\eqref{eq:2.4.4} holds as an equality in
  $X$ for all $t\in [0,T)$, what finishes the proof. 
\end{proof}

Assume that $u$ is a similarity solution in the sense of
Definition~\ref{def:simsol}, i.e.\
\begin{equation}\label{eq:2.3.1z}
  u(t)=a\bigl(g(\sigma(t))\bigr)\sol{v}
\end{equation}
for suitable functions
$\sigma\in C^1(J,\R)$, $g\in \CC^1(\sigma(J),G)$ and $\sol{v}\in Y_1$.
By writing the right hand side of \eqref{eq:2.3.1z} in the form
$a\bigl(g(\sigma(t))g(\sigma(0))^{-1}\bigr)
a\bigl(g(\sigma(0))\bigr)\sol{v}$, we may assume without loss of
generality that $u_0:=u(0)=\sol{v}$.
When we differentiate \eqref{eq:2.3.1z} and use that $u$ solves the PDE
\eqref{eq:burgersdd'}, we obtain
\[u_t=F\bigl(a(g(\sigma(t)))\sol{v}\bigr)= a\bigl(g(\sigma(t))\bigr)
T_\mathbbm{1} a\,\sol{v}[T_{g(\sigma(t))}L_{g(\sigma(t))^{-1}}
g'(\sigma(t))\dot{\sigma}(t)].\]
With the symmetry property of $F$ from Proposition~\ref{prop:symburg},
this equality is equivalent to
\[0=F(\sol{v})-T_\mathbbm{1} a\,\sol{v}\Bigl[
T_{g(\sigma(t))}L_{g(\sigma(t))^{-1}}
g'(\sigma(t))\frac{\dot{\sigma}(t)}{m(g(\sigma(t)))}\Bigr].\]
Under the assumption that
$T_{\mathbbm{1}}a\,\sol{v}[\epsilon_j]$, $j=1,\ldots,\dim(G)$ are
linearly independent for any basis
$\{\epsilon_1,\ldots,\epsilon_{\dim(G)}\}$ of $\liealg{g}$, the
function $\sol{\mu}:=T_{g(\sigma(t))}L_{g(\sigma(t))^{-1}}
g'(\sigma(t))\frac{\dot{\sigma}(t)}{m(g(\sigma(t)))}$ must be
independent of $t$.  Therefore, the tuple $(\sol{v},\sol{\mu})\in Y_1\times
\liealg{g}$ solves
\[0=F(\sol{v})-T_\mathbbm{1}a\,\sol{v}[\sol{\mu}],\text{ and }
\sol{v}=u_0.\]
By Theorem~\ref{thm:keyfreeze}~\ref{itm:thm2.10.ii} the function
$\sol{u}$, given by
$\sol{u}(t)=a\bigl(\sol{g}(\sol{\sigma}(t))\bigr) \sol{v}$, where
$\sol{g}(\tau)=\exp(\sol{\mu}\tau)$ for all $\tau\ge 0$ and $\sol{\sigma}$ solves
$\dot{\sol{\sigma}}(t)=m\bigl(\sol{g}(\sol{\sigma}(t)))$,
$\sigma(0)=0$, is a solution to the Cauchy-problem 
\[u_t=F(u),\quad u(0)=u_0.\]
By uniqueness of this solution, we find $u(t)=a\bigl(g(\sigma(t))\bigr)
\sol{v}=a\bigl(\exp(\sol{\mu}\sol{\sigma}(t))\bigr) \sol{v}$.
Recollecting this discussion, we obtain the following result.
\begin{proposition}\label{prop:simsol}
  A function $u$ is a similarity solution of $u_t=F(u)$ with profile
  $\sol{v}=u(0)\in Y_1$ for which $T_\mathbbm{1}a\,\sol{v}$ is injective, 
  if and only if
  there is
  $\sol{\mu}=(\sol{\mu}_1,\sol{\mu}_2,\sol{\mu}_3)\in\liealg{g}=\R\times
  \so(d-1)\times \R^d$ with
  \[0=F(\sol{v})-T_\mathbbm{1}a\,\sol{v}[\sol{\mu}].\]
  Moreover, in this case $u(t)= a\bigl(\exp(\sol{\mu}\sol{\sigma}(t))\bigr) \sol{v}$,
  where $\sol{\sigma}$ is given by
  \[\sol{\sigma}(t)=\frac{\ln\bigl((2p-2)\sol{\mu}_1t+1\bigr)}{(2p-2)\sol{\mu}_1},\]
  which solves
  $\dot{\sigma}=m\bigl(\exp(\sol{\mu}\,\sigma)\bigr)=
  e^{(2-2p)\sol{\mu}_1\sigma}$, $\sigma(0)=0$.
\end{proposition}
\begin{remark}\label{rem:nonSim}
  Note that if $p\neq\frac{d+1}{d}$ there exist no similarity solutions with a
  localized finite mass profile $\sol{v}$ and nontrivial scaling $\sol{\mu}_1$.
  That is, $\sol{v}$ and $\sol{\mu}_1$ satisfy
  $\sol{v}\in H^2(\R^d)\cap L^\infty(\R^d)$,
  $|\gradient \sol{v}(x)|+|\sol{v}(x)|^p\le \const |x|^d$,
  $\int_{\R^d} \sol{v}(x)\,dx\ne 0$, and $\sol{\mu}_1\ne 0$. 

  This follows since by Proposition~\ref{prop:simsol} also the function $u$,
  given by $u(t)=a\bigl(\exp(\sol{\mu}\,\sol{\sigma}(t))\bigr)\sol{v}$ has the
  decay properties of $\sol{v}$, so that Gau\ss{}' Theorem on the one hand shows
  \[\frac{d}{dt}\int_{\R^d}u(x,t)\,dx=0.\]
  On the other hand, 
  $p>1$ and $\sol{\mu}_1\ne 0$ imply
  $\sol{\alpha}(t)=\bigl((2p-2)\sol{\mu}_1t+1\bigr)^{\frac{1}{2p-2}}\ne 1$ by
  Proposition~\ref{prop:simsol}, so that
  \[\begin{aligned}\int_{\R^d} u(x,t)\,dx &= \int_{\R^d}
    \alpha(t)^{-1}\sol{v}\bigl(\sol{\alpha}(t)^{1-p}\wt{Q}(t)\bigl(x-b(t)\bigr)\,dx
  =\int_{\R^d} \alpha(t)^{dp-p-1}\sol{v}(y)\,dy.\end{aligned}\]
  Then $\int_{\R^d}u(x,t)\,dx$ is constant if and only if $dp-d=1$.

  Also, we can look at this from taking the point of view of the solution in the
  co-moving coordinates, i.e.\ $v$ given by \eqref{eq:2.4.aa}.  
  Because of the divergence theorem and the assumption of localization
  (sufficiently fast decay of $|v(\xi)|$ and $|\gradient v(\xi)|$ as $|\xi|\to
  \infty$) this satisfies the identity
  \[\frac{d}{d\tau} \int_{\R^d} v(\xi,\tau)\,d\xi = \int_{\R^d}
    \bigl(F(v)-T_{\mathbbm{1}}a\,v[\mu]\bigr)(\xi,\tau)\,d\xi 
    = \mu_1(\tau) (1+d-dp)\int_{\R^d}v(\xi,\tau)\,d\xi.\]
  Therefore, the mass of $v$ satisfies under the above assumptions of
  localization the equation
  \begin{equation}\label{eq:nonSim}
    \int_{\R^d}v(\xi,\tau)\,d\xi =
    e^{(1+d-dp)\int_0^\tau\mu_0(\eta)\,d\eta}\int_{\R^d}v(\xi,0)\,d\xi.
  \end{equation}
\end{remark}

\begin{remark}
  The results from Theorem~\ref{thm:keyfreeze} and
  Proposition~\ref{prop:simsol} are not restricted to Burgers' type
  equations, but hold for all evolution equations which possess a
  similar symmetry structure.
\end{remark}
\begin{example}
  We now continue Example~\ref{ex:gen} and explicitly state the \emph{co-moving
  equation}~\eqref{eq:2.4.1} and the \emph{reconstruction
  equations}~\eqref{eq:recbasic} for the cases of $d=1,2,3$ spatial dimensions.
  This is needed for the actual implementation of the freezing method in the
  end.

  To enhance readability of the equations, we as usual denote elements in $G$ by
  $g=(\alpha,Q,b)$ with $\alpha\in\R_+$, $Q\in \SO(d-1)$, $b\in \R^d$
  and elements in $\liealg{g}=T_\mathbbm{1}G$ are denoted by
  $\mu=(\mu_1,\mu_2,\mu_3)$, where $\mu_1\in\R$,
  $\mu_2=\sum_{j=1}^{\dim\so(d-1)}\mu_2^j S_j\in \so(d-1)$ with
  $S_1,\ldots,S_{\dim\so(d-1)}$ a basis of $\so(d-1)$, $\mu_2$
  is identified with
  $(\mu_2^1,\dots,\mu_2^{\dim(\so(d-1))})^\top\in\R^{\dim(\so(d-1))}$, and $\mu_3=\sum_{j=1}^d
  \mu_3^j \e_j\in\R^d$.  In particular, we have 
  \begin{equation*}
    (\mu_1,0,0)=\mu_1\epsilon_1,\;
    (0,\mu_2,0)=\sum_{j=1}^{\dim\so(d-1)} \mu_2^j \epsilon_{1+j},\;
    (0,0,\mu_3)=\sum_{j=1}^d \mu_3^j \epsilon_{1+\dim\so(d-1)+j},
  \end{equation*}
  where $\epsilon_1,\ldots,\epsilon_{\dim\liealg{g}}$ is the canonical
  basis of $\liealg{g}$ which was introduced in Example~\ref{ex:gen} in the
  cases $d=1,2,3$.
  \begin{enumerate}[label=\textup{(\roman*)}]
    \item \label{itm:ex:freeze1}
      For $d=1$ we first assume that $\wh{T}>0$ and
      $\mu\in \CC([0,\wh{T});\liealg{g})$ are given.  The generators
      $T_\mathbbm{1}a\,v[\epsilon_1]$ and $T_\mathbbm{1}a\,v[\epsilon_2]$ are
      calculated in Example~\ref{ex:gen}~\ref{itm:ex:gen1d} and inserting them
      into the \emph{co-moving equation}~\eqref{eq:2.4.1} we obtain the
      explicit form
      \begin{subequations}\label{eq:2.5.2}
        \begin{equation}\label{eq:2.5.2a}
          v_\tau=\nu v_{xx}-\tfrac{1}{p}\bigl(|v|^p\bigr)_x
          +\mu_1\bigl((p-1)(xv)_x+(2-p)v\bigr)+\mu_3 v_x.
        \end{equation}
        Moreover, the functions $g$ and $\sigma$ can be obtained from the
        \emph{reconstruction equations}~\eqref{eq:thm.2.1} and
        \eqref{eq:thm.2.2}.
        With the help of \eqref{eq:GdLmult} and \eqref{eq:defm} these take the
        explicit form
        \begin{align}\label{eq:2.5.2b}
          g'&=\vect{\alpha\\b}'=
          T_\mathbbm{1}L_{(\alpha,b)}\vect{\mu_1\\\mu_3}=\vect{\alpha
          \mu_1\\\alpha \mu_3},\quad \alpha(0)=1,\,b(0)=0,\\
          \label{eq:2.5.2c}
          \dot{\sigma}&=m\bigl(g(\sigma)\bigr)
          =m\bigl(\alpha(\sigma),b(\sigma)\bigr)=\alpha(\sigma)^{2-2p},
          \quad\sigma(0)=0.
        \end{align}
      \end{subequations}
      From Theorem~\ref{thm:keyfreeze}~\ref{itm:thm2.10.ii} then follows that if
      $v\in\CC([0,\wh{T});Y_1)\cap \CC^1([0,\wh{T});X)$, 
      a solution of the original equation \eqref{eq:burgersdd'} is obtained 
      by formula \eqref{eq:2.4.aa}.
    \item \label{itm:ex:freeze2}
      For the case $d=2$ we proceed similar and obtain by using the generators
      calculated in Example~\ref{ex:gen}~\ref{itm:ex:gen2d} the co-moving
      equation
      \begin{subequations}\label{eq:2.5.4}
        \begin{equation}\label{eq:2.5.4a}
          v_\tau=\nu \laplace v
          -\tfrac{1}{p}\bigl(|v|^p\bigr)_x 
          + \mu_1(p-1)\bigl((xv)_x+(yv)_y\bigr)+\mu_1(3-2p)v
          + \mu_3^1 v_x+ \mu_3^2 v_y.
        \end{equation}
        Moreover, the reconstruction equations \eqref{eq:thm.2.1} and
        \eqref{eq:thm.2.2} become
        \begin{align}
          g'&=\vect{\alpha\\b}'=
          T_{\mathbbm{1}}L_{g}\vect{\mu_1\\\mu_3}
          =\vect{\alpha \mu_1\\\alpha \mu_3},\quad \alpha(0)=1,\,
          b(0)=\vect{0\\0},\label{eq:2.5.4b}\\
          \dot{\sigma}&=m\bigl(g(\sigma)\bigr) =
          \alpha(\sigma)^{2-2p},\quad\sigma(0)=0.\label{eq:2.5.4c} 
        \end{align}
      \end{subequations}
    \item\label{itm:ex:freeze3}
      Similar considerations yield for $d=3$ the co-moving equation
      \begin{subequations}\label{eq:2.5.7} 
        \begin{equation}\label{eq:2.5.7a}
          \begin{aligned}
            v_\tau=&\nu\laplace v - \tfrac{1}{p}\bigl(|v|^p)_{x}
            +\mu_1(p-1)\bigl((xv)_x+(yv)_y+(zv)_z\bigr)
            +\mu_1(4-3p) v\\
            &\qquad+\mu_2^1\bigl((yv)_z-(zv)_y\bigr)+\mu_3^1v_x+\mu_3^2 v_y
            +\mu_3^3v_z
          \end{aligned}
        \end{equation}
        and the reconstruction equations
        \begin{align}\label{eq:2.5.7b}
          g'&=\vect{\alpha\\Q\\b}'=
          T_\mathbbm{1}L_{g} \vect{\mu_1\\\mu_2^1
      \vect{0&-1\\1&0}\\\mu_3} =
      \vect{\alpha \mu_1\\
        \mu_2^1Q\vect{0&-1\\1&0}\\
      \alpha \mu_3},\; 
      \vect{\alpha\\Q\\b}(0)=\vect{1\\I\\0},\\
      \dot{\sigma}&=m\bigl(g(\sigma)\bigr) =
      \alpha(\sigma)^{2-2p},\quad \sigma(0)=0.\label{eq:2.5.7c} 
    \end{align}   
  \end{subequations}
  Note, that $Q\in \SO(2)$ is of the form
  $Q=\vect{\cos(\phi)&-\sin(\phi)\\ \sin(\phi)&\cos(\phi)}$, so that the
  differential equation \eqref{eq:2.5.7b} for $Q$ yields the equations
  \[\tfrac{d}{d\tau} \cos(\phi)=-\sin(\phi)\phi'= -\mu_2^1\sin(\phi),\;
  \tfrac{d}{d\tau} \sin(\phi)=\cos(\phi)\phi'=\mu_2^1\cos(\phi).\]
  Therefore, $\phi'=\mu_2^1$ completely describes the evolution of
  the $Q$-component of $g$ and using this equation, it is implicit that the
  solution to \eqref{eq:2.5.7b} always stays on the Lie-group $G$.
\end{enumerate} 
\end{example}

\subsection{Phase Conditions}
Theorem~\ref{thm:keyfreeze} relates the Cauchy-problem for
\eqref{eq:burgersdd}
\begin{equation}\label{eq:OCauchy}
  \left\{\begin{aligned}
    u_t&=\nu\laplace u-\frac{1}{p}\dd{x_1}\left(|u|^p\right)=:
    F(u),\\
    u(0)&=u_0,
  \end{aligned}
  \right.
\end{equation}
to the Cauchy-problem in the new, time-dependent coordinate system.
Roughly speaking, we can rephrase the result as follows:\\
A function $u\in\CC([0,T);Y_1)\cap\CC^1([0,T);X)$ solves \eqref{eq:OCauchy}
if and only if
the functions $v\in\CC([0,\wh{T});Y_1)\cap \CC^1([0,\wh{T});X)$,
$\mu\in\CC([0,\wh{T});\liealg{g})$, $g\in\CC^1([0,\wh{T});G)$,
$\sigma\in \CC^1([0,T);[0,\wh{T}))$ solve the system
\begin{equation}\label{eq:CoCauchyO}
  \begin{aligned}
    v_\tau&=F(v)-T_{\mathbbm{1}}av[\mu(\tau)],&v(0)&=u_0,\\
    g_\tau&=T_{\mathbbm{1}}L_{g(\tau)}[\mu(\tau)],&g(0)&=\mathbbm{1},\\
    \bigl(\sigma^{-1}\bigr)_\tau&=\frac{1}{m\bigl(g(\tau)\bigr)},
    &\sigma(0)&=0
  \end{aligned}
\end{equation}
and the functions $u$ and $v,g,\sigma$ are related by the identity
$u(t)=a\bigl(g(\sigma(t))\bigr)v\bigl(\sigma(t)\bigr)$.

Although \eqref{eq:OCauchy} is a well-posed problem, e.g. in $\R^3$
for initial data in $H^2$ (see Remark~\ref{rem:wellposed}), the
system \eqref{eq:CoCauchyO} is not well-posed, due to an ambiguity in
the choice of $\mu(\tau)$.  More precisely, the variable $\mu$
introduces $\dim\liealg{g}$ additional degrees of freedom to the
system.  As is standard in the numerical freezing method, see
e.g.~\cite{BeynThuemmler:2004}, we therefore introduce $\dim\liealg{g}$
additional algebraic equations, so called phase conditions, to cope
with this ambiguity.  In this article we restrict to two specific
choices of phase conditions, which work very well in our numerical
experiments in Section~\ref{sec:exp}.
We will actually use integral phase conditions which were originally
introduced by \cite{Doedel:1981} for the numerical approximation of
periodic orbits.  Thus we are limited to the case $X=L^2$ with
$Y_1=H^2_\scal$, and we assume the well-posedness of the problem in
these spaces from now on.  We note in passing that by using weighted
integrals, it is not difficult to allow similar phase conditions also for
the $X=\CC^0_0$ and $Y_1=\CC^2_\scal$ case.

\textbf{Type 1: Orthogonal phase condition}\\
The idea of the orthogonal phase condition is, to require that the
time-evolution of $v$ is always $L^2$-{orthogonal} to its group orbit.
This amounts to requiring 
\begin{equation}\label{eq:OPhaseO}
  0=\langle v_\tau,T_{\mathbbm{1}}av[\epsilon_j]\rangle,\quad
  j=1,\ldots,\dim\liealg{g}, 
\end{equation}
where $\{\epsilon_1,\ldots,\epsilon_{\dim\liealg{g}}\}$ is a basis of
$\liealg{g}$.  Inserting the $v$-equation of \eqref{eq:CoCauchyO} into
\eqref{eq:OPhaseO} yields
\begin{equation}\label{eq:OPhaseI}
  0=\langle
  T_{\mathbbm{1}}av[\epsilon_j],F(v)-T_{\mathbbm{1}}av[\mu]\rangle=:
  \Psi^{\mathrm{orth}}_j(v,\mu),\quad
  j=1,\ldots,\dim\liealg{g},
\end{equation}
which we call the ``\emph{orthogonal phase condition}'' and, in fact, is a
linear equation for the algebraic variable $\mu$ if $v$ is known.
Assuming the invertibility of the matrix
$\Bigl(\langle
T_{\mathbbm{1}}av[\epsilon_j],T_{\mathbbm{1}}av[\epsilon_k]
\rangle_{j,k=1\ldots\dim\liealg{g}}\Bigr)
  \in
\R^{\dim\liealg{g},\dim\liealg{g}}$,
\eqref{eq:OPhaseI} can easily be solved for the algebraic
variables $\mu$,
\[\mu=\Bigl(\langle
T_{\mathbbm{1}}av[\epsilon_j],
T_{\mathbbm{1}}av[\epsilon_k]\rangle_{j,k=1\ldots\dim\liealg{g}}\Bigr)^{-1} 
\langle
T_{\mathbbm{1}}av[\epsilon_j],F(v)\rangle_{j=1\ldots\dim\liealg{g}}.\]
In principle it is possible to insert this formula directly into
\eqref{eq:CoCauchyO}, but we rather supplement \eqref{eq:CoCauchyO} with
\eqref{eq:OPhaseI}, since we are anyway also interested in the actual value of
$\mu$ as an important constant of motion, as already argued in the introduction.

\textbf{Type 2: Fixed phase condition}\\
The idea of the fixed phase condition is, to require that the
$v$-component of the solution always lies in a fixed,
$\dim\liealg{g}$-co-dimensional hyperplane, which is given as the
level set of a fixed, linear mapping, i.e.
\begin{equation}\label{eq:FPhase}
  0=\psi_j(v) - r_j=:\Psi_j^{\mathrm{fix}}(v),\quad j=1,\ldots,\dim\liealg{g},
\end{equation}
where $\psi_j\in X^\ast$, $j=1,\ldots,\dim\liealg{g}$, are linearly
independent elements of the dual of $X$ and $r_j\in\R$.  A standard
choice, cf.~\cite{BeynThuemmler:2004}, is the following:  Assume
that there is a ``suitable'' reference function $\wh{u}$ given and
then one require that the $v$-component of the solution always
satisfies
\begin{equation}\label{eq:FPhaseS}
  0=\langle T_{\mathbbm{1}}a\wh{u}[\epsilon_j],v-\wh{u}\rangle,\quad
  j=1,\ldots,\dim\liealg{g}, 
\end{equation}
where $\{\epsilon_1,\ldots,\epsilon_{\dim\liealg{g}}\}$ is a basis of
$\liealg{g}$.  
\begin{remark}
  The phase condition \eqref{eq:FPhaseS} can also be obtained by
  requiring that the $v$-component of the solution is always better
  aligned in the $L^2$-norm to $\wh{u}$ than to any other element of
  the group orbit of $\wh{u}$, i.e.
  \begin{equation}\label{eq:L2best}
    \argmin_{g\in G}\bigl\| a(g)\wh{u}-v\bigr\|^2_{L^2}=\mathbbm{1}.
  \end{equation}
  A necessary condition for \eqref{eq:L2best} is \eqref{eq:FPhaseS}.
\end{remark}

We now augment system
\eqref{eq:CoCauchyO} with one of the phase conditions
\eqref{eq:OPhaseI} or \eqref{eq:FPhaseS} and obtain the PDAE system of the
numerical freezing method
\begin{subequations}\label{eq:CoCauchy}
  \begin{align}
    \label{eq:CoCauchyA1}
    v_\tau&=F(v)-T_{\mathbbm{1}}av[\mu(\tau)],&v(0)&=u_0,\\
    \label{eq:CoCauchyA2}
    0&=\Psi(v,\mu),&&\\
    \label{eq:CoCauchyB}
    g_\tau&=T_{\mathbbm{1}}L_{g(\tau)}[\mu(\tau)],&g(0)&=\mathbbm{1},\\
    \label{eq:CoCauchyC}
    \bigl(\rho\bigr)_\tau&=\frac{1}{m\bigl(g(\tau)\bigr)},
    &\rho(0)&=0,
  \end{align}
\end{subequations}
where $\Psi(v,\mu)$ is either $\Psi^{\mathrm{orth}}$ from
\eqref{eq:OPhaseI} or $\Psi^{\mathrm{fix}}$ from \eqref{eq:FPhase}, and
we denote $\rho=\sigma^{-1}$.

\begin{remarks}\label{rem:3.2}
  \begin{enumerate}
    \item Note that the ordinary differential equations \eqref{eq:CoCauchyB}
     and \eqref{eq:CoCauchyC} actually decouple from
     \eqref{eq:CoCauchyA1} and \eqref{eq:CoCauchyA2} and hence could be
     solved in a post-processing step.
   \item
     Observe that \eqref{eq:CoCauchy} consists of the PDE
     \eqref{eq:CoCauchyA1} which has a hyperbolic-parabolic structure,
     coupled to a system of ordinary differential equations
     \eqref{eq:CoCauchyB} and \eqref{eq:CoCauchyC} and coupled to a
     system of algebraic equations \eqref{eq:CoCauchyA2}.  Moreover, in
     \eqref{eq:CoCauchyA1} the hyperbolic part dominates for $\nu<<1$ and
     the parabolic part dominates for $\nu$ sufficiently large.   
   \item
     For the choice \eqref{eq:OPhaseI} the system \eqref{eq:CoCauchyA1},
     \eqref{eq:CoCauchyA2} is a partial differential algebraic equation
     of ``time-index'' 1 and for the choice \eqref{eq:FPhaseS}
     it is of ``time-index'' 2.  Here we understand the index as a
     differentiation index (see \cite{MartinsonBarton:2000}).
 \end{enumerate}
\end{remarks}


%% file: sec5.tex
\section{Numerical Results}\label{sec:exp}
In this section we now present the result of several numerical experiments.  We
use a numerical second order scheme that we develop in \cite{Rottmann:2016b}.
Here we do not go into the details of the numerical scheme but only mention that
it is based on a central method-of-lines system for hyperbolic conservation
laws, adapted from \cite{KurganovTadmor:2000} to \eqref{eq:CoCauchy}, and then
fully discretized with an IMEX-Runge-Kutta time-discretization in the spirit of
\cite{AscherRuuthSpiteri:1997} to cope with the different parts of the
equation.  For details we refer to \cite{Rottmann:2016b}.  

To distinguish between the original coordinates and the coordinates of the
freezing method, we always denote original space by $x\in\R^d$ and the
original time by $t$ and the original solution in these coordinates is denoted
by $u$, whereas $\xi\in\R^d$ and $\tau\ge 0$ denote the space and
time in the new coordinates and $v$ is the solution in these new coordinates.

\subsection{1d-Experiments}
We choose the spatial step size
$\Delta \xi=0.01$ and the time step size $\Delta \tau$ is a CFL-based multiple
of $\Delta \xi$, see \cite{Rottmann:2016b}.
In all 1d-experiments we choose the initial condition given by
\begin{equation}\label{eq:1dinit}
  u_0(x)=\begin{cases}\sin(2x),&-\frac{\pi}{2}\le x\le 0,\\
    \sin(x),&0\le 0\le \pi,\\
    0,&\text{otherwise}.\end{cases}
\end{equation}
Note that this belongs to $H^1(\R)$ but not to $H^2(\R)$.  In our experiments we
observed that the method and our numerical scheme work just well also for
piecewise continuous initial data.
Moreover, if nothing else is stated, we
choose the fixed phase condition, where we choose the initial condition as a
suitable reference function.  In case the solution evolves too far away from this
reference function, we update it with the current state of the solution.

\textbf{Variation of the parameter $p$.}
In our first series of computations we consider the behavior of the freezing
method for the Cauchy-problem for Burgers' equation
\[u_t=0.4 u_{xx}-\tfrac{1}{p}\partial_x\bigl(|u|^p\bigr),\quad u(0)=u_0\]
where we choose different values for the parameter $p$.  
\begin{figure}[t!]
  \centering
  \subfigure[]{\includegraphics[height=5cm]%
    {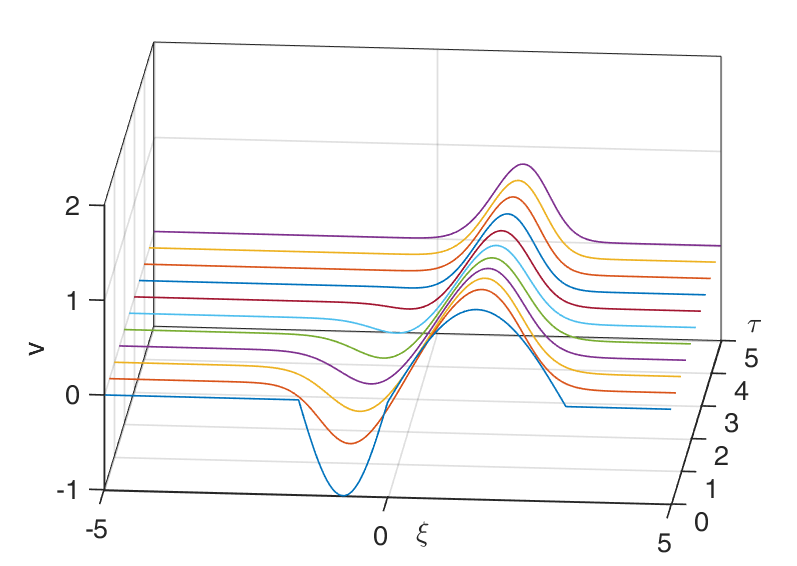}\label{fig:exp01_1}}
  \subfigure[]{\includegraphics[height=5cm]%
    {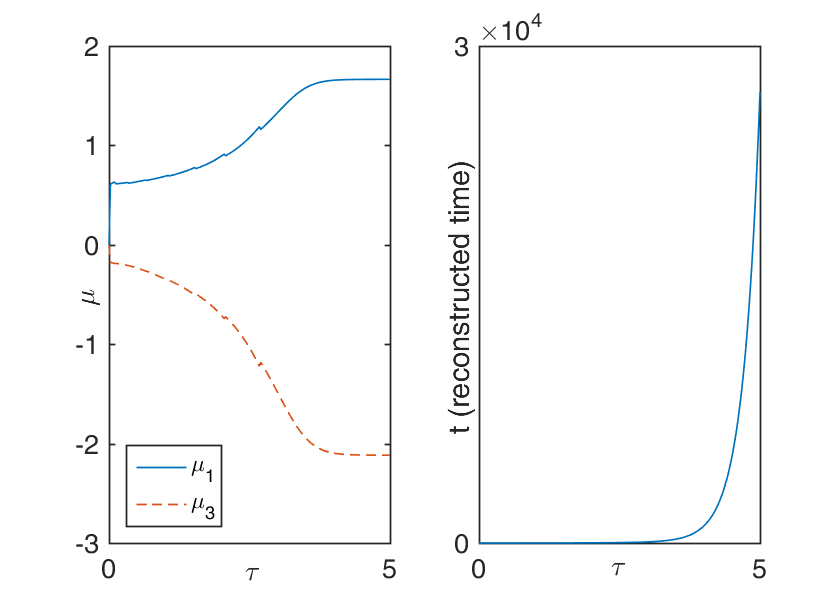}\label{fig:exp01_2}}
  \caption{Plots of the time evolution for conservative 1d-Burgers equation
    ($p=2$) with viscosity
    $\nu=0.4$ (a) in the scaled (computational) coordinates at
    different time instances, (b) the evolution
    of the variables in the Lie-algebra and the evolution of the
  original time as function of the scaled time $t(\tau)$.}
  \label{fig:nu04_evoscaled}
\end{figure} 
In Fig.~\ref{fig:nu04_evoscaled} we show the results for the freezing method
for the conservative 1d Burgers' equation (i.e. $p=2$).  One can very well
observe that the solution to the freezing PDAE \eqref{eq:CoCauchyA1},
\eqref{eq:CoCauchyA2} stabilizes as $\tau$ increases.  Moreover, also the
algebraic variables $\mu_1$ (scaling) and $\mu_3$ (spatial velocity) converge to
constant values as $\tau$ increases.  We also calculate the solution to
the reconstruction equations \eqref{eq:CoCauchyB}, \eqref{eq:CoCauchyC} and use
these to obtain the solution $u$ in the original coordinates, given by
formula \eqref{eq:2.4.aa} and shown
Fig.~\ref{fig:nu04_evoorig}.
\begin{figure}[b!]
  \centering
  \subfigure[]{\includegraphics[height=5cm]%
    {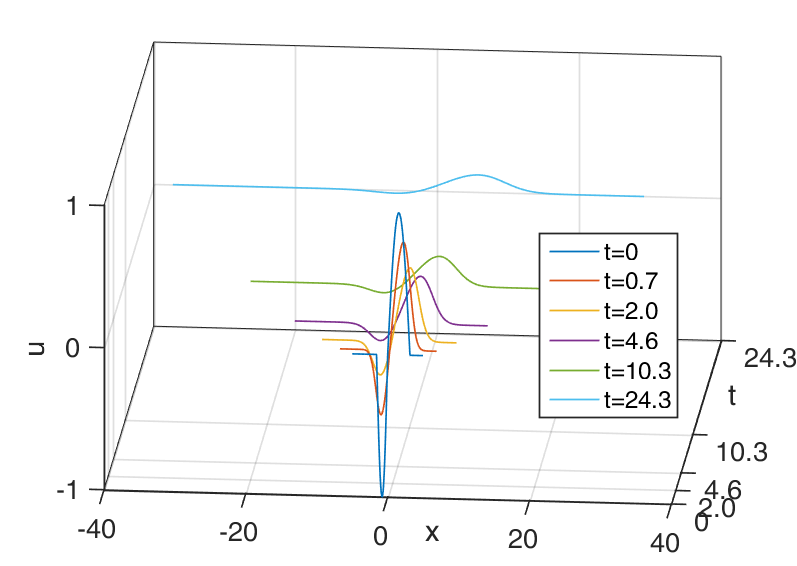}\label{fig:exp01_3a}}
  \subfigure[]{\includegraphics[height=5cm]%
    {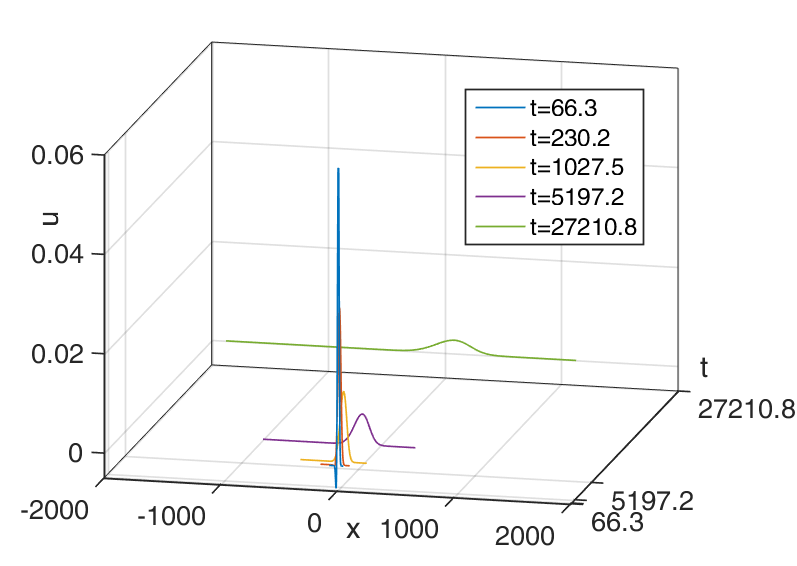}\label{fig:exp01_3b}}
  \caption{Reconstructed solutions of the 1d-Burgers'
    equation ($p=2$) with viscosity $\nu=0.4$.  Initial time interval $[0,24.3]$
  (a) and later time interval $[66.3,27210.8]$ (b).}
  \label{fig:nu04_evoorig}
\end{figure} 
The time-instances are precisely
$t=t(\tau)$ with the $\tau$ from Fig.~\ref{fig:exp01_1}.
As is well-known, in the original coordinates
the solution stabilizes to the
constant zero as time tends to infinity.
Note that the original time $t(\tau)$ increases very rapidly with $\tau$
(e.g.\ $t(5)\approx 27210.8$) and one needs a very large time-domain if the
solution is calculated in the original coordinates.

We repeat the above numerical experiment for the Cauchy problem
\[u_t=u_{xx}-\tfrac{2}{3}\partial_x\bigl(|u|^\frac{3}{2}\bigr),\quad u(0)=u_0,\]
where we changed the parameter $p$ to $p=\tfrac{3}{2}$.
The result of this computation is presented in Fig.~\ref{fig:exp02}.
\begin{figure}[t!]
  \centering
  \subfigure[]{\includegraphics[height=3.6cm]%
    {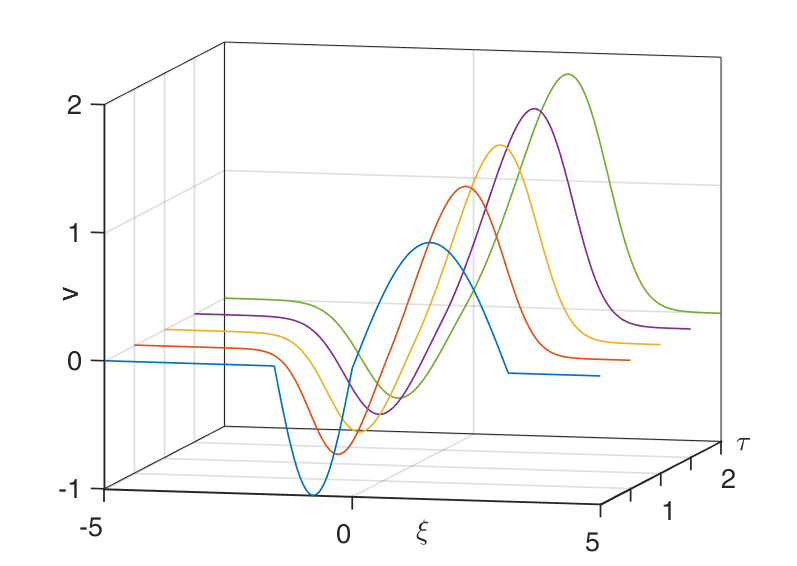}\label{fig:exp02_1}}
  \subfigure[]{\includegraphics[height=3.6cm]%
    {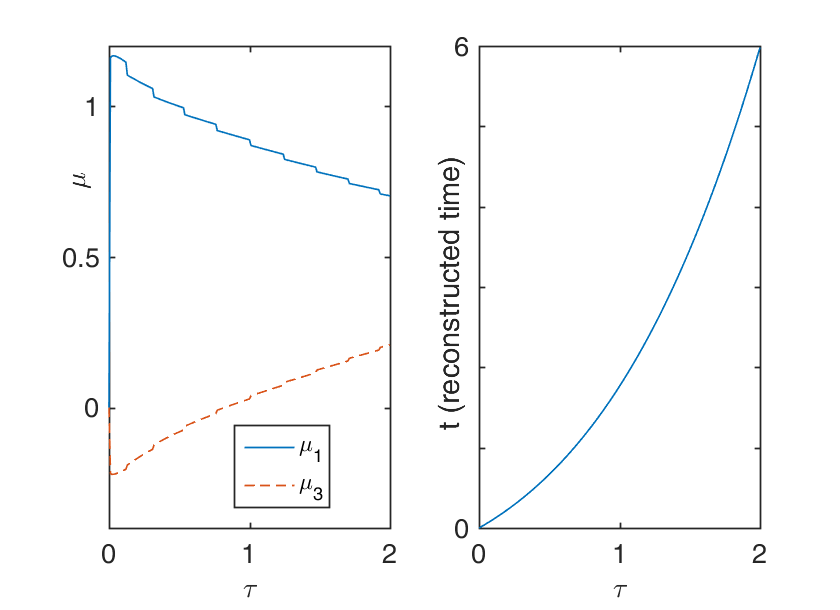}\label{fig:exp02_2}}
  \subfigure[]{\includegraphics[height=3.6cm]%
    {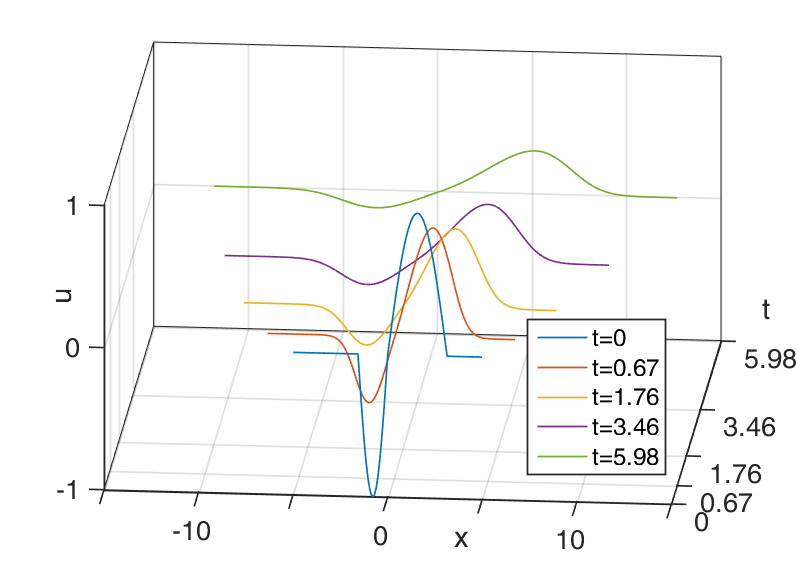}\label{fig:exp02_3}}
  \caption{Time evolution of 1d-Burgers' equation
    with $p=\tfrac{3}{2}$ and viscosity $\nu=0.4$.
    In (a) several time-instances of the numerical solution of
    the freezing method are shown, (b) shows how the algebraic
    variables $\mu_1$ (scaling) and $\mu_3$ (spatial velocity) depend on the
    scaled time $\tau$, and (c) shows the solutions from (a) 
   in the original coordinates.}
  \label{fig:exp02}
\end{figure} 
One can nicely observe that the
mass of the profile $v$ in the new coordinates grows as $\tau$ increases, which
is in accordance with Remark~\ref{rem:nonSim} and formula \eqref{eq:nonSim}
since
$p<\tfrac{d+1}{d}$ and $\mu_1>0$.  Nevertheless, we can still doe the freezing
method calculations on a fixed bounded domain and obtain the solution in the
original coordinates from the reconstruction equations.  The result is shown in
Fig.~\ref{fig:exp02_3}.

We also repeat the experiment with $p=\tfrac{5}{2}$.  From
Remark~\ref{rem:nonSim} and formula \eqref{eq:nonSim} we now expect that the
profile $v$ in the co-moving coordinates decays to zero as $\tau$ tends to
infinity.
\begin{figure}[b!]
  \centering
  \subfigure[]{\includegraphics[height=3.6cm]%
    {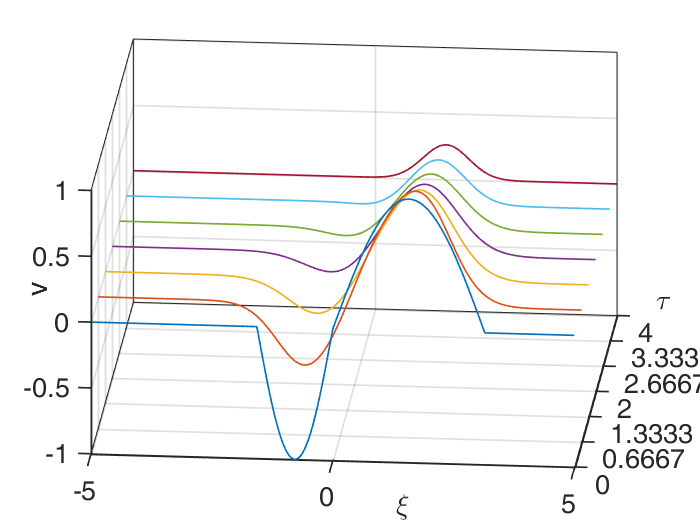}\label{fig:exp03_1}}
  \subfigure[]{\includegraphics[height=3.6cm]%
    {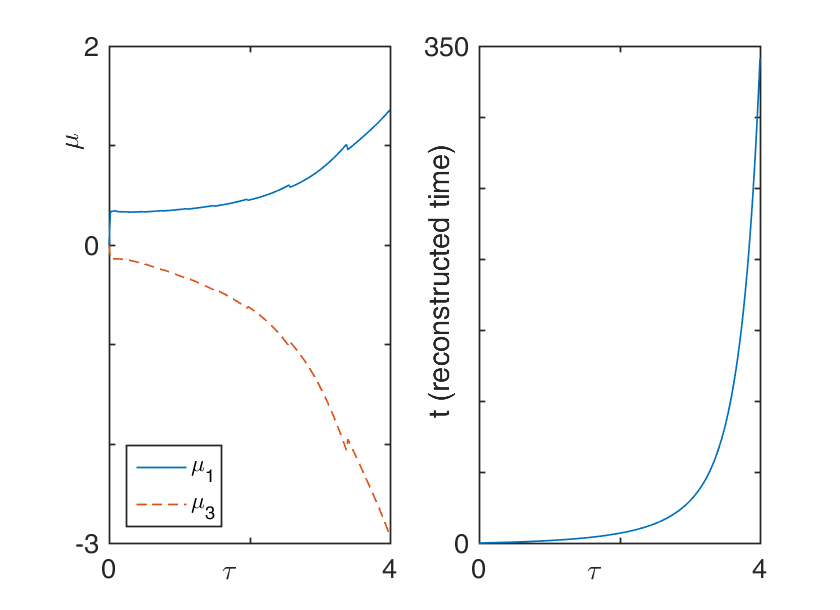}\label{fig:exp03_2}}
  \subfigure[]{\includegraphics[height=3.6cm]%
    {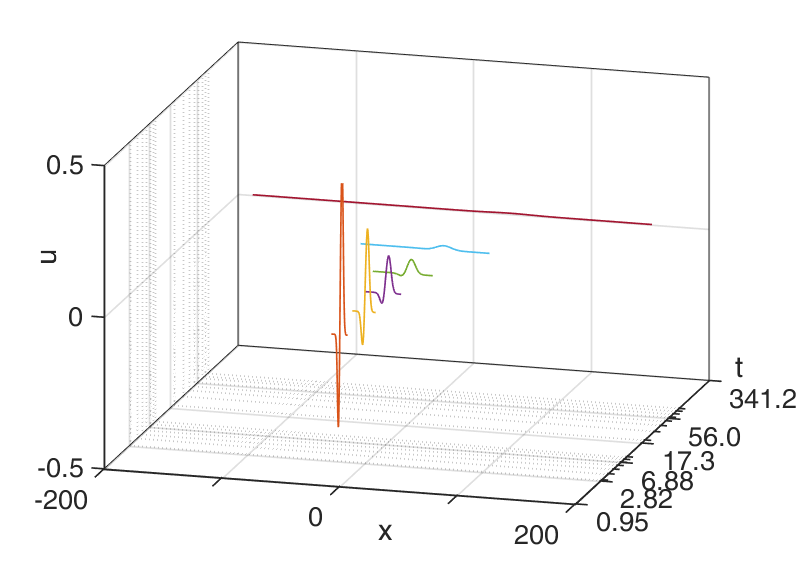}\label{fig:exp03_3}}
  \caption{Time evolution for the generalized 1d-Burgers equation
    with $p=\tfrac{5}{2}$ and viscosity $\nu=0.4$.
    In (a) the solution of the freezing method at different time-instances is
    shown, (b) shows the algebraic variables $\mu$
    and original time $t$ plotted as functions of $\tau$.
    Finally, (c) shows the solution from (a) in the original coordinates.}
  \label{fig:exp03}
\end{figure} 
The results of the simulation with the numerical freezing method are shown in
Fig.~\ref{fig:exp03} and precisely
reproduce this expectation.  Note that we actually only calculated until
$\tau=4$ in the scaled coordinates.  But $\tau=4$ corresponds to $t\approx 341$,
and in the original coordinates the profile $v(\tau=4)$, which is
calculated on the fixed domain $\xi\in[-5,5]$ corresponds to the solution in the
original coordinates on the domain $x\in [-187.44, 151.54]$.  Also
note that we have
chosen a logarithmic scale for the time-axis in Fig.~\ref{fig:exp03_3}.  We
again see that, although there is no true relative equilibrium of the equation
$u_t=u_{xx}-\tfrac{2}{5}\partial_x\bigl(|u|^{\frac{5}{2}}\bigr)$, the freezing
method allows us to do a calculation on a fixed bounded domain for a far longer
time, than it would be possible for the original problem.

\textbf{Different phase conditions.}
In all the experiments performed so far, we have chosen the fixed
phase condition given by \eqref{eq:FPhaseS}.  We now compare the results we
obtain
with the freezing method by using the fixed phase condition with the results we
obtain when we use the orthogonal phase condition.
\begin{figure}[h!t]
  \centering
  \subfigure[]{\includegraphics[height=5cm]%
    {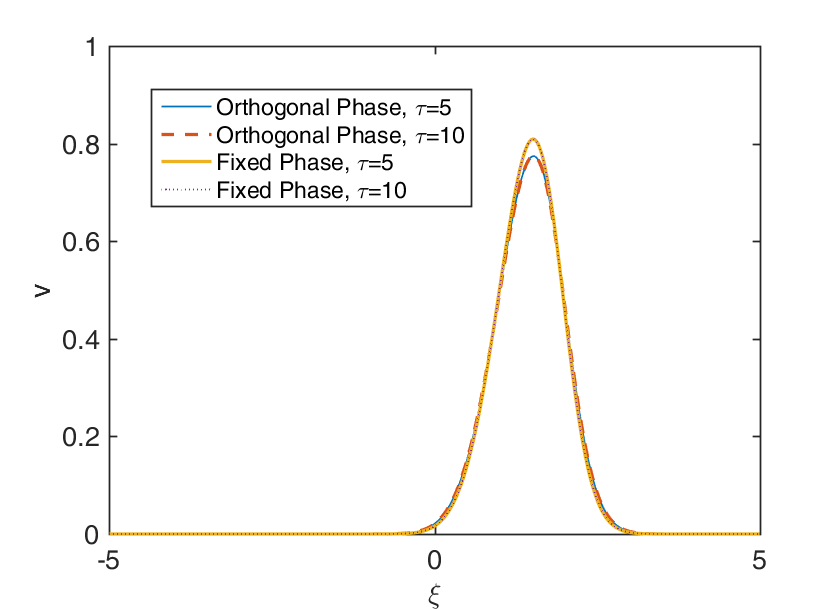}\label{fig:exp05_1}}
  \subfigure[]{\includegraphics[height=5cm]%
    {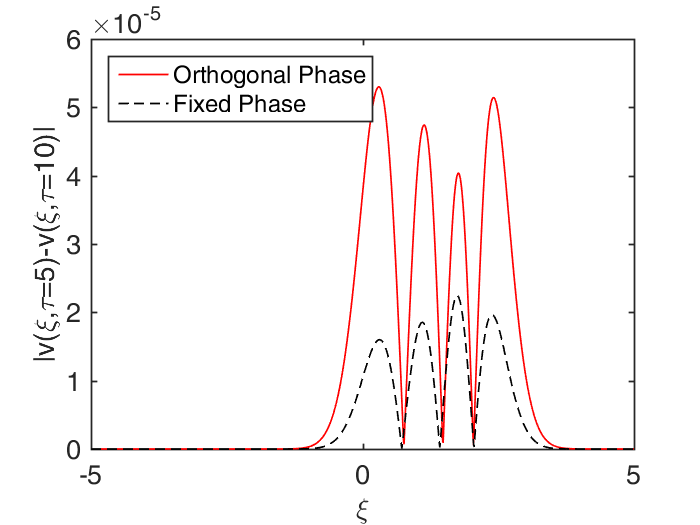}\label{fig:exp05_2}}
  \caption{Comparison of the final states for different phase conditions.
    In (a) the solutions obtained with the freezing method for the
    orthogonal and fixed phase conditions at
    $\tau=5$ and $\tau=10$ are shown.  The solutions for the same phase
    conditions do not change from $\tau=5$ to $\tau=10$.  In (b) we plot the
    actual difference of
    $|v(\xi,5)-v(\xi,10)|$. }
  \label{fig:exp05}
\end{figure} 
We again consider the conservative Burgers' equation, i.e.\ $p=2$, and choose
the viscosity $\nu=0.4$.  In Fig.~\ref{fig:exp05} we show the solution to the
freezing method obtained with the fixed phase condition at $\tau=5$ and
$\tau=10$ and also the solution obtained with the orthogonal phase condition at
the same time-instances $\tau=5$ and $\tau=10$.  In Fig.~\ref{fig:exp05_1}
these are plotted in one diagram and the solutions with the same
phase conditions but at different time-instances virtually do not differ and,
therefore, seem to be constant rest states.
The difference of the solutions to the same
phase condition but at different times is shown in Figure~\ref{fig:exp05_2}.
As is obvious from Fig.~\ref{fig:exp05_1}, these steady states do depend on the
choice of the phase condition and, moreover, we even obtain different limits for
the algebraic values.  In the fixed
phase condition case we obtain $\mu_1=1.665$ (scaling) and $\mu_3=-2.117$
(translation), in the orthogonal phase condition case we obtain $\mu_1=1.524$
and $\mu_3=-1.931$.
\begin{figure}[h!t]
  \centering
  \subfigure[]{\includegraphics[height=3.6cm]%
    {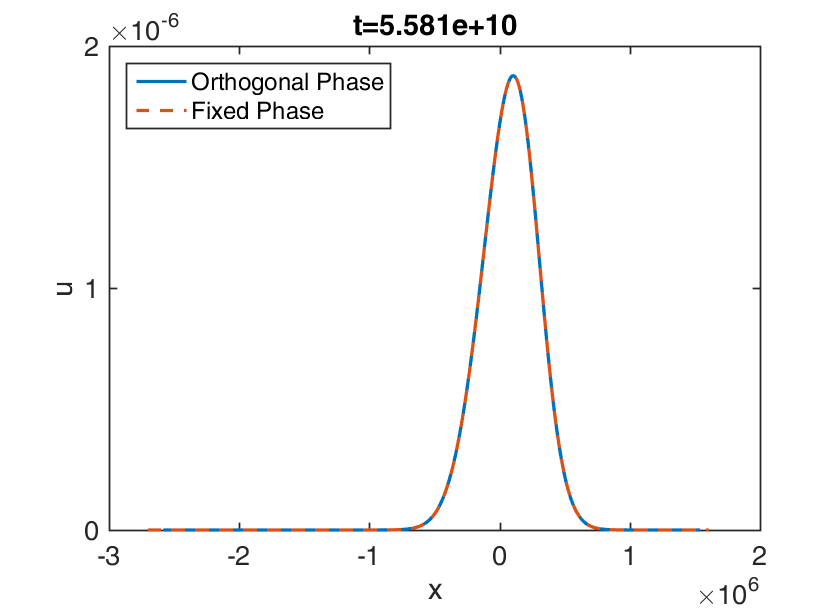}\label{fig:exp05_3}}
  \subfigure[]{\includegraphics[height=3.6cm]%
    {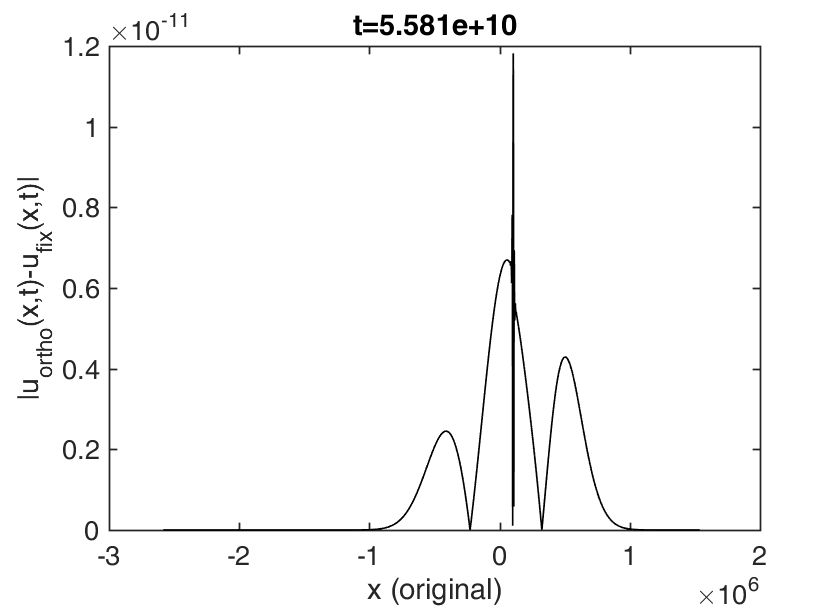}\label{fig:exp05_3b}}
  \subfigure[]{\includegraphics[height=3.6cm]%
    {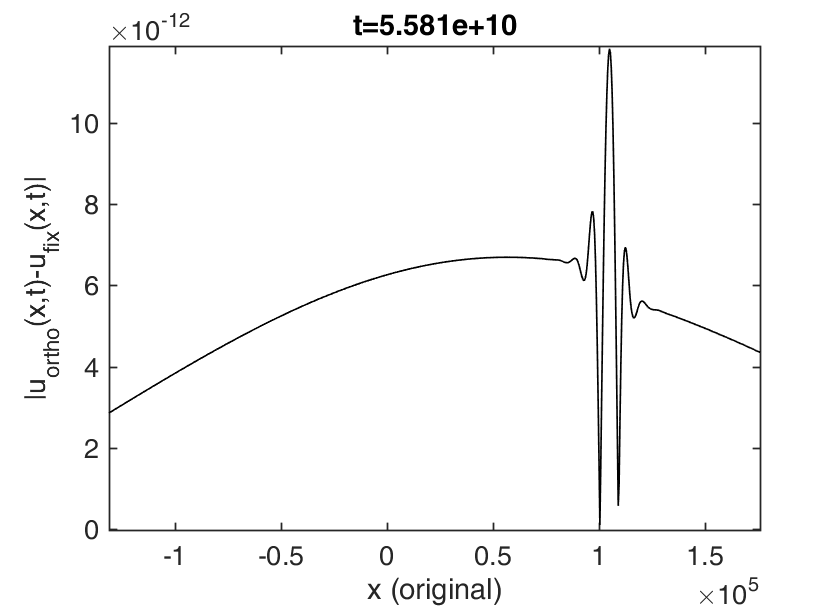}\label{fig:exp05_3c}}
    \caption{Difference of the solutions obtained at $t=5.581\cdot 10^{10}$ via
    the freezing method with orthogonal and fixed phase condition. In (a) the
  two different solutions are shown, (b) and (c) shows there difference.}
  \label{fig:exp05b}
\end{figure} 
Nevertheless, the solution in the original coordinates at the latest common
original time is plotted in Fig.~\ref{fig:exp05b} and shows only a very small
difference between the two different choices of phase conditions as plotted in
Fig.~\ref{fig:exp05_3b} and Fig.~\ref{fig:exp05_3c}.

\textbf{Metastable behavior.}
Our method directly enables us to observe the metastable
behavior in Burgers' equation with small viscosity which was first
numerically observed and analyzed in \cite{KimTzavaras:2001} and later
discussed from a dynamical systems point of view in \cite{BeckWayne:2009}.  
\begin{figure}[htb!]
  \centering
  \subfigure[]{\includegraphics[width=.49\textwidth]%
    {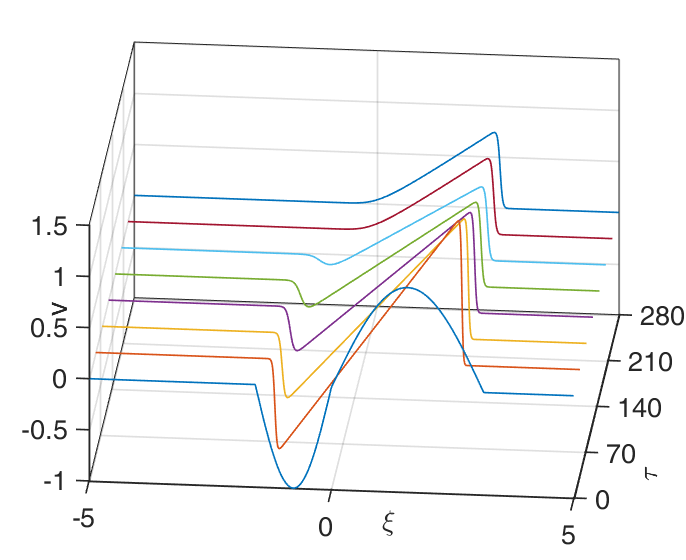}\label{fig:ETa}}
  \subfigure[]{\includegraphics[width=.49\textwidth]%
    {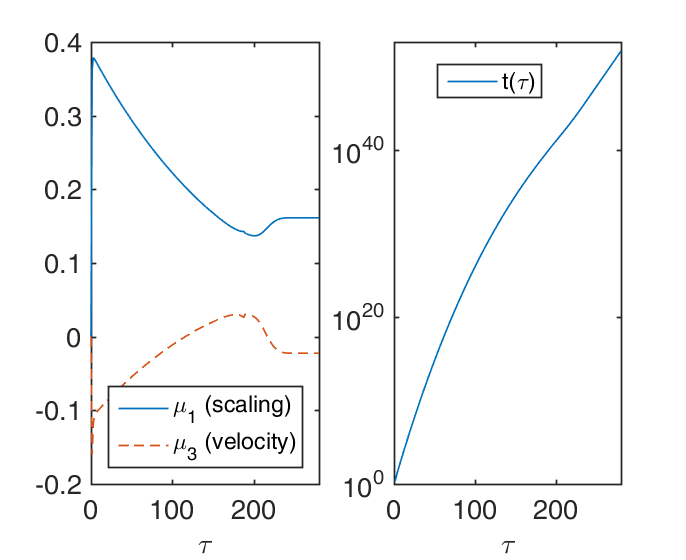}\label{fig:ETaa}}
  \caption{Time evolution for 1d-Burgers' equation, with viscosity
    $\nu=0.01$ in the scaled (computational) coordinates (a) and
    the evolution of the algebraic variables $\mu$ and of the
  original time as function of the scaled time $t(\tau)$.}
  \label{fig:ET}
\end{figure} 
Our results are presented in Fig.~\ref{fig:ET} for viscosity $\nu=0.01$.
Note that there is a very long transient, when the solution pretty much
looks like an N-wave (see Fig.~\ref{fig:ETa} until $\tau\approx 200
\hat{=}t\approx 10^{41}$) and then in a final stage converges to the true
similarity solution, which is a viscosity wave.  This convergence can
also be observed by looking at the time evolution of the algebraic variable
$\mu$, which evolves slowly first and finally stabilizes to a 
constant value.

\subsection{2d-Experiments}
We also apply the method to the
two-dimensional generalized Burgers' equations
\[\partial_t u=\nu \Delta
u-\tfrac{1}{p}\partial_x\bigl(|u|^p\bigr).\]
Again we use the numerical scheme introduced in
\cite{Rottmann:2016b}.  Moreover, we choose the
orthogonal phase condition, which is slightly more efficient than the fixed
phase condition.

\textbf{Metastable behavior.}
\begin{figure}[h!tb]
  \centering
  \subfigure[]{\includegraphics[width=.65\textwidth]%
    {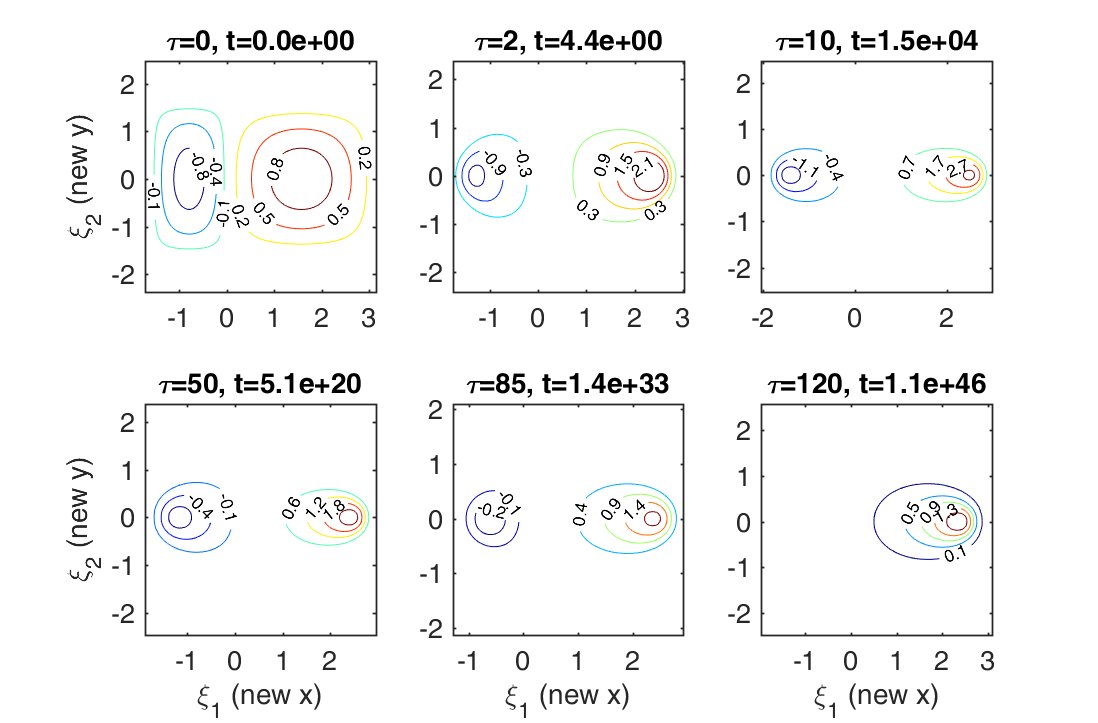}\label{fig:2d01}}%
  \subfigure[]{\includegraphics[width=.33\textwidth,height=5cm]%
    {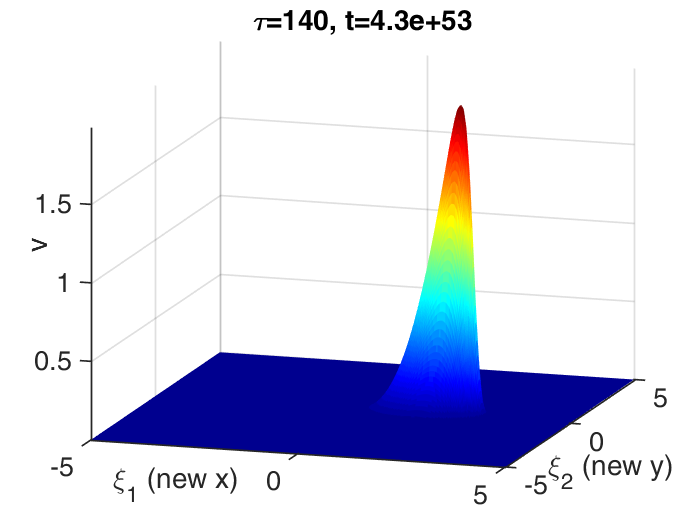}\label{fig:2d01_1}}
  \caption{Time Evolution of the 2d-Burgers' Equation in co-moving
    coordinates for $\nu=0.05$.
    In (a) we plot the contour lines at different time instances (in
    the scaled time) and 
    in (b) we show the solution at the final time $\tau=140$ ($\hat{=}$
    $t\approx 4.3\cdot 10^{53}$).}
  \label{fig:2d01a}
\end{figure}
First we consider the conservative 2-dimensional Burgers' equation, i.e.\
$p=\tfrac{3}{2}$ with very small viscosity $\nu=0.05$.  For the actual
computation with the freezing method we choose the computational domain
$\xi\in[-5,5]\times[-5,5]$ and no-flux boundary conditions.
The spatial step sizes are $\Delta \xi_1=\Delta
\xi_2=\tfrac{1}{15}$ and the time step size $\Delta\tau$ is CFL-base multiple of
these, see \cite{Rottmann:2016b}.
In Fig.~\ref{fig:2d01} we present contour plots of the solution at different
time instances and one observes that rapidly a pattern evolves which resembles
the 1d pattern of an N-wave.  This pattern exists for a very long
time until the ``negative blob'' vanishes and a final steady state is reached.
This happens approximately at $\tau=110$ ($\hat{=} t=1.69\cdot 10^{42}$).  
A plot of this final state is shown in Fig.~\ref{fig:2d01_1}.
\begin{figure}[b]
  \centering
  \subfigure[]{\includegraphics[width=.33\textwidth]%
    {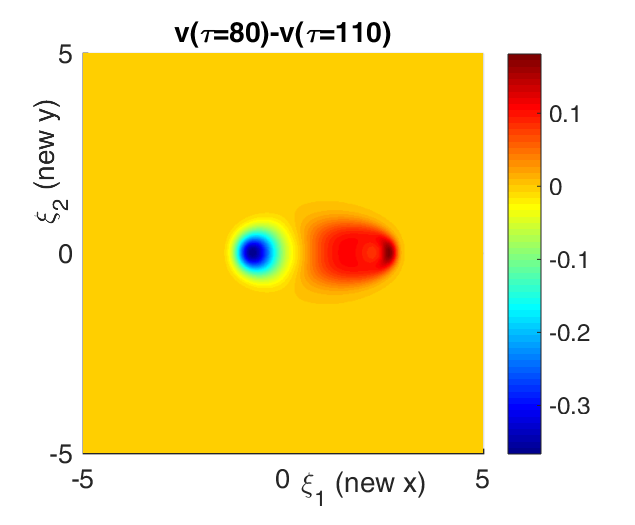}\label{fig:2d01_3a}}%
  \subfigure[]{\includegraphics[width=.33\textwidth]%
    {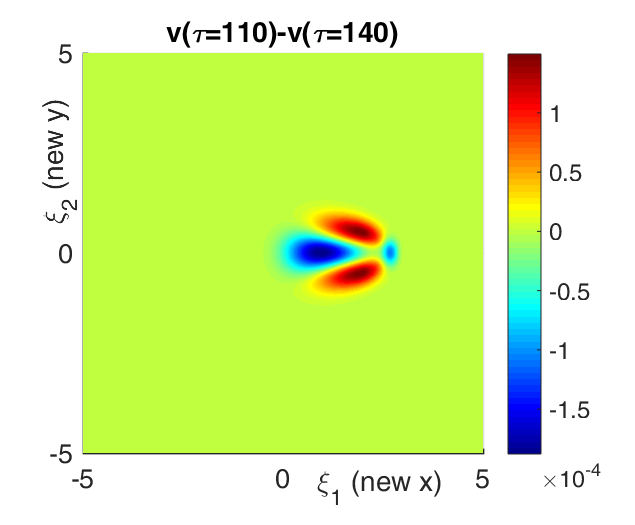}\label{fig:2d01_3b}}%
  \subfigure[]{\includegraphics[width=.33\textwidth]%
    {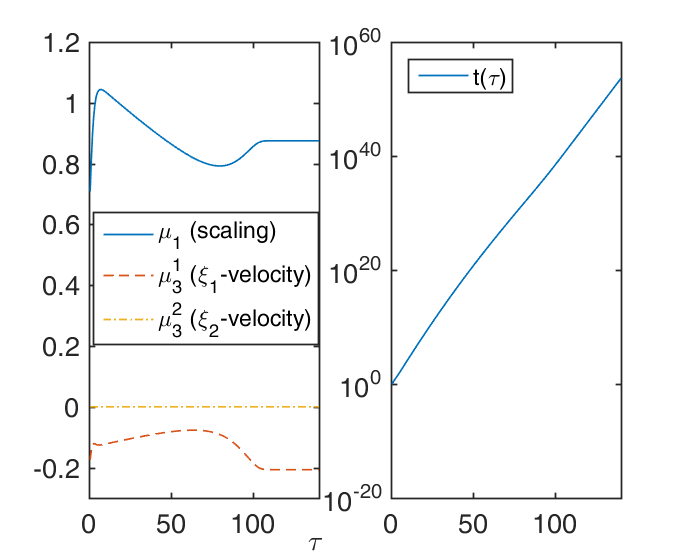}\label{fig:2d01_2}}%
  \caption{Convergence to relative equilibrium: In (a) and (b) 
    we plot the difference of the solutions to the freezing equations for
    different times $\tau$.  In (c) the evolution of the 
    evolution of the algebraic variables and the
  original time is shown.} %
  \label{fig:2d01b}
\end{figure}
Fig.~\ref{fig:2d01b} shows that this profile indeed is a steady state:
The difference of the solution to the freezing
method at $\tau=80$ ($\hat{=}t=2.60\cdot 10^{31}$) and $\tau=110$
($\hat{=}t=1.69\cdot 10^{42}$) is plotted in Fig.~\ref{fig:2d01_3a}. Note that
there is a large negative
area where the solutions differ the most (actually their values differ by
approximately $-0.37$).  This area corresponds to the ``negative blob'' which
vanishes as $\tau$ increases (Fig.~\ref{fig:2d01}).  When comparing the
solutions at $\tau=110$ ($\hat{=}t=1.69\cdot 10^{42}$) and $\tau=140$
($\hat{=}t=4.27\cdot 10^{53}$), this spot indeed has vanished and the difference
of these two solutions is only of order $10^{-4}$.  Also the 
algebraic variables, shown in Fig.~\ref{fig:2d01_2},
slowly evolve until $\tau\approx 110$, when they finally stabilize.

\textbf{Variation of the parameter $p$.}
We finally consider the long-time behavior of Burgers' equation with fixed
viscosity $\nu=0.4$ and vary the parameter $p$.
For the experiments we choose the initial condition
\[u_0(x,y)=\begin{cases} \sin(2x)\sin(y),& -\tfrac{\pi}{2}<x<0,\;0<y<\pi,\\
    \sin(x)\cos(y),&0<x<\pi,\;-\tfrac{\pi}{2}<y<\tfrac{\pi}{2},\\
  0,&\text{otherwise},\end{cases}\]
which is depicted in Fig.~\ref{fig:2d02_init}.

For the conservative Burgers' equation, i.e.\ $p=\tfrac{3}{2}$, we again observe
the convergence to a steady state of the freezing equations and hence to a
relative equilibrium.  
\begin{figure}[t!b]
  \centering
  \subfigure[]{\includegraphics[width=.33\textwidth]%
    {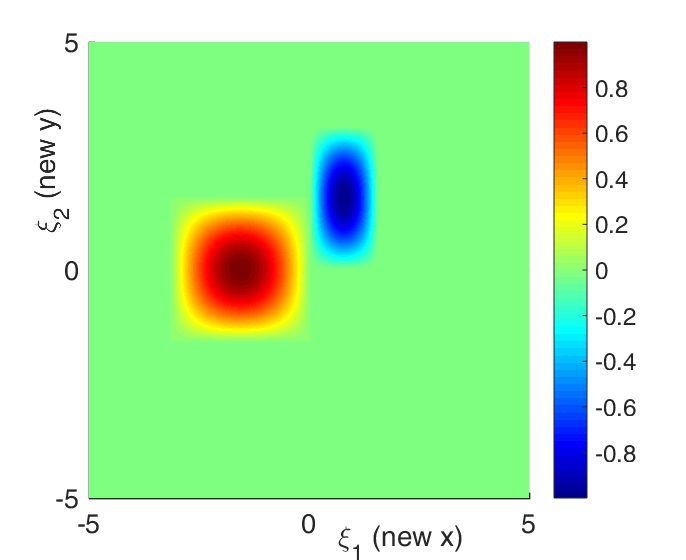}\label{fig:2d02_init}}%
  \subfigure[]{\includegraphics[width=.33\textwidth]%
    {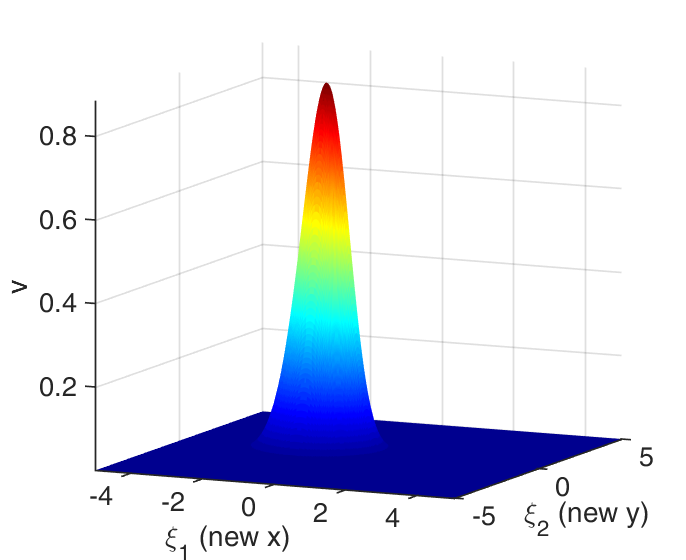}\label{fig:2d02_a}}%
  \subfigure[]{\includegraphics[width=.33\textwidth]%
    {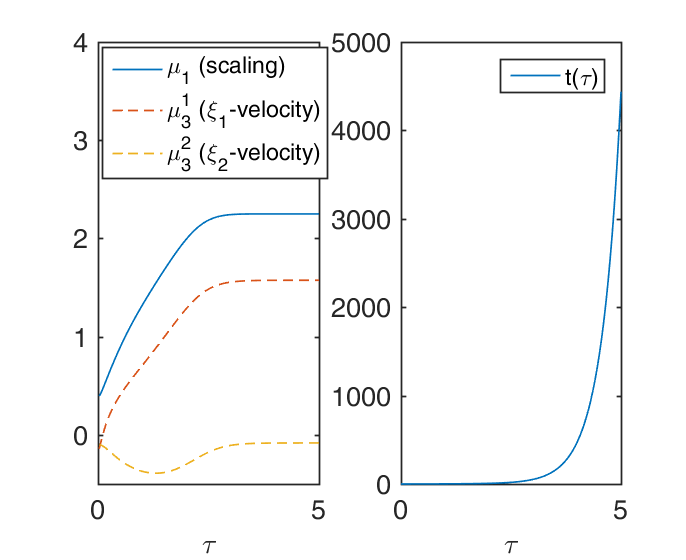}\label{fig:2d02_c}}%
    \caption{Freezing method for $u_t=0.4\laplace
      u-\tfrac{2}{3}\partial_x \bigl(|u|^{\tfrac{3}{2}}\bigr)$:
      (a) shows the initial condition, (b) the steady state obtained at
      $\tau=5$, and (c) shows the evolution of the algebraic variables and
    original time as functions of $\tau$.}
    \label{fig:2d02}
\end{figure}
In Figure~\ref{fig:2d02} we plot the final state of the
calculation at $\tau=5$ which corresponds to $t\approx4436.7$.  Note that the
final state has a non-constant $\xi_2$-velocity, which is due to
the fact that the function does not have a center of mass along the
$\xi_1$-axis.

%
For the standard 2d-Burgers' equation, that is $p=2$, we obtain
that the solution $(v,\mu)$ of the freezing method does not converge to a
localized steady state but the total mass of $v$ decays as predicted
by Remark~\ref{rem:nonSim}.  Nevertheless, it is possible and does make sense to
compute the solution to the original equation by using the freezing method on a
fixed computational domain.  
The final state at $\tau=5$ ($\hat{=}t=11$) and its reconstruction to the
original coordinates are shown in Fig.~\ref{fig:2d04}.
\begin{figure}[htb]
  \centering
  \subfigure[]{\includegraphics[width=.33\textwidth]%
    {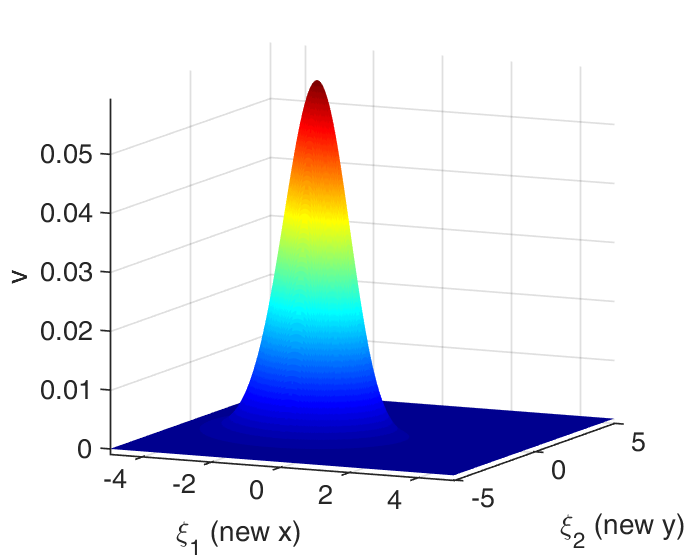}\label{fig:2d04_a}}%
  \subfigure[]{\includegraphics[width=.33\textwidth]%
    {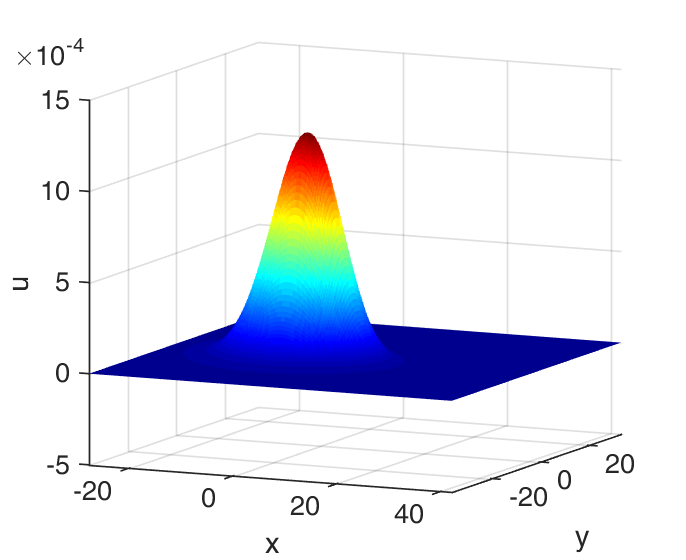}\label{fig:2d04_b}}%
  \subfigure[]{\includegraphics[width=.33\textwidth]%
    {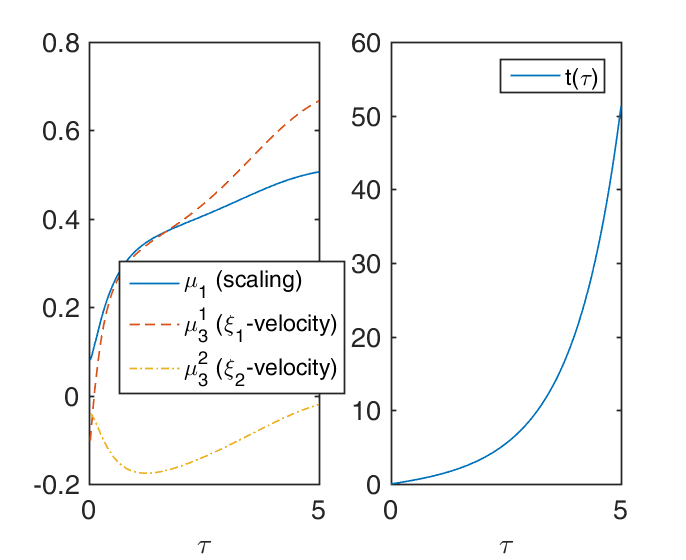}\label{fig:2d04_c}}%
    \caption{Solution of the freezing method at $\tau=5$ for the 2d-Burgers'
      equation $u_t=0.4\laplace u-\tfrac{1}{2}\partial_x\bigl(|u|^2\bigr)$.  In
    (a) in the computational coordinates, in (b) reconstruction of the solution,
  and evolution of the algebraic variables in (c).}\label{fig:2d04}
\end{figure}


%% file: similarity.bbl
\begin{thebibliography}{10}

\bibitem{AscherRuuthSpiteri:1997}
U.~M. Ascher, S.~J. Ruuth, and R.~J. Spiteri.
\newblock Implicit-explicit {R}unge-{K}utta methods for time-dependent partial
  differential equations.
\newblock {\em Appl. Numer. Math.}, 25(2-3):151--167, 1997.
\newblock Special issue on time integration (Amsterdam, 1996).

\bibitem{BecKhanin:2007}
J.~Bec and K.~Khanin.
\newblock Burgers turbulence.
\newblock {\em Phys. Rep.}, 447(1-2):1--66, 2007.

\bibitem{BeckWayne:2009}
M.~Beck and C.~Wayne.
\newblock {Using global invariant manifolds to understand metastability in the
  Burgers equation with small viscosity.}
\newblock {\em SIAM J. Appl. Dyn. Syst.}, 8(3):1043--1065,, 2009.

\bibitem{BeynOttenRottmann:2013}
W.-J. {Beyn}, D.~{Otten}, and J.~{Rottmann-Matthes}.
\newblock {Stability and computation of dynamic patterns in PDEs.}
\newblock In {\em {Current challenges in stability issues for numerical
  differential equations. Lectures of the CIME-EMS summer school, Cetraro,
  Italy, June 2011}}, pages 89--172. Cham: Springer; Firenze: Fondazione CIME,
  2014.

\bibitem{BeynThuemmler:2004}
W.-J. Beyn and V.~Th{\"u}mmler.
\newblock Freezing solutions of equivariant evolution equations.
\newblock {\em SIAM J. Appl. Dyn. Syst.}, 3(2):85--116 (electronic), 2004.

\bibitem{Burgers:1948}
J.~M. Burgers.
\newblock A mathematical model illustrating the theory of turbulence.
\newblock {\em Advances in Applied Mechanics}, 1948.

\bibitem{Doedel:1981}
E.~Doedel.
\newblock A{UTO}: a program for the automatic bifurcation analysis of
  autonomous systems.
\newblock In {\em Proceedings of the {T}enth {M}anitoba {C}onference on
  {N}umerical {M}athematics and {C}omputing, {V}ol. {I} ({W}innipeg, {M}an.,
  1980)}, volume~30, pages 265--284, 1981.

\bibitem{Elstrodt:2009}
J.~{Elstrodt}.
\newblock {\em {Ma\ss- und Integrationstheorie.}}
\newblock Berlin: Springer, 6th corrected ed. edition, 2009.

\bibitem{Hopf:1950}
E.~Hopf.
\newblock The partial differential equation {$u_t+uu_x=\mu u_{xx}$}.
\newblock {\em Comm. Pure Appl. Math.}, 3:201--230, 1950.

\bibitem{KimTzavaras:2001}
Y.~J. Kim and A.~E. Tzavaras.
\newblock Diffusive {$N$}-waves and metastability in the {B}urgers equation.
\newblock {\em SIAM J. Math. Anal.}, 33(3):607--633 (electronic), 2001.

\bibitem{KurganovTadmor:2000}
A.~Kurganov and E.~Tadmor.
\newblock New high-resolution central schemes for nonlinear conservation laws
  and convection-diffusion equations.
\newblock {\em J. Comput. Phys.}, 160(1):241--282, 2000.

\bibitem{MartinsonBarton:2000}
W.~S. Martinson and P.~I. Barton.
\newblock A differentiation index for partial differential-algebraic equations.
\newblock {\em SIAM J. Sci. Comput.}, 21(6):2295--2315, 2000.

\bibitem{Olver:1979}
P.~J. Olver.
\newblock {Symmetry groups and group invariant solutions of partial
  differential equations.}
\newblock {\em {J. Differ. Geom.}}, 14:497--542, 1979.

\bibitem{Olver:1986}
P.~J. Olver.
\newblock {\em Applications of {L}ie groups to differential equations}, volume
  107 of {\em Graduate Texts in Mathematics}.
\newblock Springer-Verlag, New York, 1986.

\bibitem{Pazy:1983}
A.~Pazy.
\newblock {\em Semigroups of linear operators and applications to partial
  differential equations}, volume~44 of {\em Applied Mathematical Sciences}.
\newblock Springer-Verlag, New York, 1983.

\bibitem{Rauch:1991}
J.~{Rauch}.
\newblock {\em {Partial differential equations.}}
\newblock New York etc.: Springer-Verlag, 1991.

\bibitem{Rottmann:2016b}
J.~Rottmann-Matthes.
\newblock {An IMEX-RK scheme for the Method of Freezing at the Example of
  Burgers' Equation}.
\newblock In preparation.

\bibitem{Rottmann:2010}
J.~Rottmann-Matthes.
\newblock {\em {Computation and Stability of Patterns in Hyperbolic-Parabolic
  Systems}}.
\newblock {Shaker Verlag}, Aachen, 2010.
\newblock {{PhD} thesis, Bielefeld University}.

\bibitem{Rottmann:2012a}
J.~Rottmann-Matthes.
\newblock {Stability and Freezing of Nonlinear Waves in First Order Hyperbolic
  PDEs}.
\newblock {\em J. Dynam. Differential Equations}, 24(2):341--367, 2012.

\bibitem{RowleyKevrekidisMarsdenLust:2003}
C.~W. Rowley, I.~G. Kevrekidis, J.~E. Marsden, and K.~Lust.
\newblock Reduction and reconstruction for self-similar dynamical systems.
\newblock {\em Nonlinearity}, 16(4):1257--1275, 2003.

\bibitem{Rudin:1991}
W.~{Rudin}.
\newblock {\em {Functional analysis. 2nd ed.}}
\newblock New York, NY: McGraw-Hill, 2nd ed. edition, 1991.

\bibitem{SandstedeScheelWulff:1999}
B.~{Sandstede}, A.~{Scheel}, and C.~{Wulff}.
\newblock {Bifurcations and dynamics of spiral waves.}
\newblock {\em {J. Nonlinear Sci.}}, 9(4):439--478, 1999.

\bibitem{Ziemer:1989}
W.~P. Ziemer.
\newblock {\em Weakly differentiable functions}, volume 120 of {\em Graduate
  Texts in Mathematics}.
\newblock Springer-Verlag, New York, 1989.
\newblock Sobolev spaces and functions of bounded variation.

\end{thebibliography}
